\documentclass[10pt]{amsart}
\usepackage{amsmath,amsthm,amsfonts,amscd,amssymb,amscd,
			bbm,bm,bookmark,bxcjkjatype,
			caption, comment,
			doi,
			enumerate,enumitem,esint,extpfeil,
			fancyhdr,
			graphicx,
			hyperref,
			mathrsfs,mathtools,
			subcaption,
			textcomp,
			xcolor}
\newcommand{\fdiv}{\mathrm{div}} 
\renewcommand{\div}{\mathrm{div}\,}
\newcommand{\pdiv}[1]{\fdiv\left( #1 \right)} 
\newcommand{\init}[0]{\mathrm{in}} 
\newcommand{\bR}{\mathbb{R}}
\newcommand{\cP}{\mathcal{P}}
\newcommand{\ac}[0]{\mathrm{ac}}

\newcommand{\cl}[1]{\overline{#1}}
\newcommand{\sE}{\mathscr{E}}
\newcommand{\sG}{\mathscr{G}}
\newcommand{\sF}{\mathscr{F}}
\newcommand{\bS}{\mathbb{S}}

\newcommand{\sD}{\mathscr{D}}

\newcommand{\sJ}{\mathscr{J}}
\newcommand{\cO}{\mathcal{O}}

\newcommand{\bN}{\mathbb{N}}
\newcommand{\sK}{\mathscr{K}}
\newcommand{\seq} [1]{\left\{ #1 \right\}}
\newcommand{\ol}[1]{\overline{#1}}

\newcommand{\wh}[1]{\widehat{#1}}

\newcommand{\fd}[2]{\frac{d #1}{d #2}} 			

\newcommand{\pseq} [1]{\left( #1 \right)}
\newcommand{\abs}[1]{\left\vert #1 \right\vert}
\newcommand{\set}[1]{\left\{ #1 \right\}}
\DeclareMathOperator{\supp}{supp}

\newcommand{\lap}{\Delta}  
\DeclareMathOperator*{\argmin}{arg\,min}
\newcommand{\norm}[1]{\left\Vert #1 \right\Vert}

\newcommand{\pd}[2]{\partial_{#2} #1}
\newcommand{\ul}[1]{\underline{#1}}

\def\wk{\mathrm{-wk}}
\DeclarePairedDelimiter\ceil{\lceil}{\rceil}

\usepackage[greek,english]{babel}
\usepackage{tikz}
\usepackage{float}

\def\eps{\varepsilon}
\def\pa{\partial}
\def\na{\nabla}

\numberwithin{equation}{section}

\newtheorem{theorem}{Theorem}[section]
\newtheorem{theorem*}{Theorem}

\newtheorem{lemma}[theorem]{Lemma}

\newtheorem{proposition}[theorem]{Proposition}
\newtheorem{corollary}[theorem]{Corollary}

\theoremstyle{definition}

\newlist{wellPosednessHypotheses-f}{enumerate}{1}
\newlist{refinedGrowthCondition-f}{enumerate}{1}
\newlist{singularLimitHypotheses-f}{enumerate}{1}

\setlist[wellPosednessHypotheses-f]{label = \bfseries(F\arabic*)}
\setlist[refinedGrowthCondition-f]{ label=\bfseries(F\arabic*$^\ast$)}
\setlist[singularLimitHypotheses-f]{ label = \bfseries(F\arabic*)}

\newcounter{growthCondition-f}

\title[Volume-preserving MCF as a singular limit of a diffusion-aggregation equation]{Volume-preserving mean-curvature flow as a singular limit\\ of a diffusion-aggregation equation}
\author{A. Mellet}
\author{M. Rozowski}

\address{Department of Mathematics, University of Maryland, College Park, MD 20742-4015, United States.}
\email{mellet@umd.edu}
\email{mrozowsk@umd.edu}

\thanks{
A. Mellet was partially supported by NSF Grant DMS-2009236 and DMS-2307342.
}

\begin{document}	
\maketitle

\begin{abstract}
The Patlak-Keller-Segel system of equations (PKS) is a classical example of aggregation-diffusion equation in which the repulsive effect of diffusion is in competition with the attractive chemotaxis term. 
Recent work on the {\it Parabolic-Elliptic PKS model} have shown that when the repulsion is modeled by a nonlinear diffusion term $\rho \na \rho^{m-1}$ with $m>2$, this competition leads to phase separation phenomena. Furthermore, in some asymptotic regime corresponding to a large population observed over a long enough time, the interface separating regions of high and low density evolves according to the Hele-Shaw free boundary problem with surface tension. 
In the present paper,  we consider the counterpart of that model, namely the {\it Elliptic-Parabolic PKS model} and we prove that the same phase separation phenomena occurs, but the motion of the interface is now described (asymptotically) by a volume-preserving mean-curvature flow.
\end{abstract}






\section{Introduction}
The following \emph{parabolic-parabolic Patlak-Keller-Segel system} 
with nonlinear diffusion is a classical example of an aggregation-diffusion equation \cite{Keller_Segel,Patlak}:
\begin{equation}\label{eq:PKS}
	\begin{cases}
		\alpha \pa_t \rho - \pdiv{\rho\nabla f'(\rho)} + \chi \div(\rho \na \phi) = 0,\quad \\ 
		\beta \pa_t \phi -\eta \Delta \phi = \rho-\sigma \phi.  
	\end{cases}
\end{equation} 
In this model, $\rho$ denotes the density of some organisms and $\phi$ denotes the concentration of some chemical (the chemoattractant) towards which the organisms are attracted ($\chi>0$).
The parameters $\alpha,\beta \ge 0$ are constants denoting the timescales for the relaxation of the organism density and chemoattractant concentration, respectively. 
The constant $\chi > 0$ is the strength of the chemoattraction; $\eta > 0$ is the diffusivity of the chemoattractant in the medium; and $\sigma> 0$ is the rate of destruction of the chemoattractant. 
The nonlinear diffusion term models the mutual repulsion of the organisms:
The derivative $f'(\rho)$ can be interpreted as an internal pressure caused by the natural incompressibility of the organisms.
A common choice of pressure law is  $f(\rho) = \frac{1}{m-1}\rho^m$ for $m>1$ 
which corresponds to degenerate nonlinear diffusion
or $f(\rho) = \rho \ln \rho$ which gives the standard linear diffusion.

The key feature of this system is the competition between the repulsive action of the nonlinear diffusion (incompressibility of the organisms) and the  attractive chemotaxis force $\na \phi$.
The fact that this competition might result in finite time blow-up of the density $\rho$ has been extensively studied:
In dimension $d$, 
 it is known that the existence of  global-in-time  weak solutions is guaranteed  whenever $m > 2 - \frac{2}{d}$ (see for example \cite{Mim17}  for the parabolic-parabolic system), but  blow-up can occur for sufficiently large initial mass of organisms $\int \rho_\init \,dx$ whenever $m = 2 - \frac{2}{d}$.  Similarly for the so-called \emph{parabolic-elliptic system} ($\beta=0$), it is known \cite{2006_Sugyiama_globalExistenceSubCritical-blowupSuperCritical, BRB} that blow-up may occur when $m<2-\frac{2}{d}$, while solutions are bounded in $L^\infty$ uniformly in time when $m>2-\frac{2}{d}$.

More recently, several works (see \cite{KMW2,M23} and  the survey \cite{KMJW}) have shown that when $m>2$, the competition 
between repulsive and attractive forces in \eqref{eq:PKS} leads to phase separation (at appropriate scales).
Phase separation, which refers to the ability of the organisms to organize themselves into segregated groups separated by sharp interfaces, is an important feature of many biological aggregation phenomena. Classical examples include animals moving as homogeneous patches (insect swarms) or the formation of membraneless organelles in eukaryotic cells.
When phase separation occurs, our focus naturally turns to the evolution of the sharp interface separating regions of high and low (or zero) organism density. More precisely, one would like to derive the equation (of a geometric nature) that governs the evolution of this interface, thus describing the collective dynamics of the organisms after aggregation.

Mathematically, the interface emerges through some singular limit which  describes regimes in which the size of the support of $\rho$ is much larger than the typical length scale of the transition layer   between regions of high and low density.
Following \cite{KMW2,M23,KMJW} we thus consider a very large initial population of organisms, which we quantify by introducing a positive $\eps\ll1$ such that
\[
\int \rho_{\init}(\bar{x})\, d\bar{x} =\eps^{-d},
\]
where $\bar x$ is the microscopic space variable.
We then rescale the microscopic space and time variables ($\bar{x}$ and $\bar{t}$, resp.)  as follows:
$$ x:=\eps \bar{x}, \qquad t:=\tau^{-1}\eps^2 \bar{t}$$
for some time scale $\tau > 0$ to be chosen later
(where $x$ and $t$ denote the new -- macroscopic -- space and time variables). When $\eps\ll1$, this scaling amounts to observing the evolution of the population from far away (with a zoom out factor of $\eps^{-1}$) and at time scale $\eps^{-2}\tau$.
The functions $\rho^\eps(t,x) :=\rho(\bar{t},\bar{x})$ and  $\phi^\eps(t,x) :=\phi(\bar{t},\bar{x})$ then solve
\begin{equation}\label{eq:PKSeps}
	\begin{cases}
		\alpha \tau^{-1}  \pa_t \rho^\eps - \pdiv{\rho^\eps \nabla f'(\rho^\eps)} + \chi \div(\rho^\eps \na \phi^\eps) = 0\quad & \mbox{ in }  (0,\infty)\times \Omega, \\
		\beta  \tau^{-1} \pa_t \phi ^\eps-\eta \Delta \phi ^\eps= \eps^{-2}\left(  \rho^\eps-\sigma \phi ^\eps\right) & \mbox{ in }   (0,\infty)\times \Omega.
	\end{cases}
\end{equation} 
We note that the bounded domain $\Omega \subset \bR^d$ has macroscopic size $\sim 1$ and is thus independent of $\eps$. 
Denoting by $n$ the outward unit normal vector to $\partial\Omega$, we supplement our equations with the following  no-flux boundary conditions
\begin{equation}\label{eq:bc}
	(-\rho^\eps\na f'(\rho^\eps)  +\chi  \rho^\eps\na \phi^\eps)\cdot n = 0, \qquad \na \phi\cdot n=0  \quad \mbox{on } \pa\Omega
\end{equation}
and initial conditions
$$ \rho^\eps(0,x)=\rho^\eps_{\init}(x), \qquad \phi^\eps(0,x) = \phi^\eps_{\init}(x) \qquad\mbox{ in } \Omega,$$
where we now have in particular
$$ \int_{\Omega} \rho^\eps_{\init}(x)\, dx=1.$$

The behavior of the solutions $\rho^\eps,\phi^\eps$  strongly depends on the ratio of time scales $\alpha/\beta$. 
We note that it is natural to assume that the typical time of relaxation for the organisms ($\alpha$) is very different from that of the chemoattractant ($\beta$), since the underlying processes are of a very different nature. A classical simplification of this model consists in taking $\beta=0$ (or more generally $\alpha\gg\beta$), effectively assuming that the chemoattractant concentration relaxes instantaneously (or very quickly) to its equilibrium 
(determined by the distribution of the organisms).
The resulting \emph{parabolic-elliptic} model is a popular model with good mathematical properties, e.g., a $2$-Wasserstein gradient flow structure.
In the present paper, we focus on the opposite regime $\alpha\ll \beta$, which  amounts to assuming that the organisms (whose motion results from the simultaneous, antagonistic effects of strong repulsion and transport in the direction of $\na \phi$) reach a quasi-steady state much faster than the diffusion time for the chemoattractant.

These two opposite regimes lead to very different limits when $\eps\to0^+$. In fact, the asymptotic behavior of the solution of \eqref{eq:PKSeps} depends on the relative values of the ratio $\alpha/\beta$ and $\eps$.
More precisely, we claim the following:
\begin{enumerate} 
	\item When $\frac\alpha \beta \gg\eps$, the critical time scale is given by $\tau = \alpha$.  With this choice for $\tau$, the evolution of $\lim_{\eps\to0^+}\rho^\eps$ is then described by a {\bf Stefan free boundary problem} whose long time limit is a characteristic function $\theta \chi_{E}$.
	\item When $\frac\alpha \beta \gg\eps$ and the initial condition is a characteristic function $\theta \chi_{E}$, 
	we must look at the evolution of $\rho^\eps$ over the longer time scale $\tau = \alpha \eps^{-1}$. The evolution of $\lim_{\eps\to0^+}\rho^\eps$ is then described by a {\bf Hele-Shaw free boundary problem with surface tension}.
	\item When $\frac \alpha \beta \ll\eps$, the critical time scale is $\tau =\beta$. At that time scale, the evolution of $\lim_{\eps\to0^+}\rho^\eps$ is described by a  {\bf volume-preserving mean-curvature flow}.
	\item When $\frac \alpha \beta = \eps$ and if the initial condition is a characteristic function $\theta \chi_{E}$, then
	the critical time scale is $\tau = \beta = \alpha \eps^{-1}$. The evolution of $\lim_{\eps\to0^+}\rho^\eps$ is then described by a {\bf Hele-Shaw free boundary problem with surface tension and kinetic undercooling}.
\end{enumerate}

Points (1) and (2) were studied in \cite{KMW2,M23,KMJW} in the particular case $\beta=0$. 
The main goal of this paper is to make precise point (3), the convergence to volume-preserving mean-curvature flow {\bf  in the particular case $\alpha=0$}. The last point (4), which can be seen as an interpolation of points (2) and (3), is discussed  in some upcoming  work.
\medskip

Taking $\chi=\eta=1$ (for simplicity) and $\alpha=0$, $\tau =\beta$ (as explained above), we are thus led to the following system (which is an \emph{elliptic-parabolic Patlak-Keller-Segel system} with nonlinear diffusion):
\begin{equation}  \label{eq:PKS0} 
	\begin{cases}
		- \pdiv{\rho^\eps \nabla f'(\rho^\eps)} + \div(\rho ^\eps\na \phi^\eps) = 0 \quad & \mbox{in }  (0,\infty)\times \Omega, \\
		\pa_t \phi^\eps -  \Delta \phi^\eps = \eps^{-2}\left(  \rho^\eps-\sigma \phi^\eps \right) & \mbox{in }  (0,\infty)\times \Omega, \\
		(-\rho^\eps\nabla f'(\rho^\eps) + \rho^\varepsilon\nabla\phi^\varepsilon) \cdot n = \nabla\phi^\varepsilon \cdot n = 0	&	\text{on } (0,\infty) \times \partial\Omega, \\
		\phi^\varepsilon(0,\cdot) = \phi_\init^\varepsilon & \text{in } \Omega, \\ 
		\rho \ge 0 &\text{in } (0,\infty)\times\Omega, \\
		\int_\Omega \rho^\eps(t,x)\, dx = \int_\Omega \rho^\eps(0,x)\, dx =1	&	\text{for all } t \ge 0.
	\end{cases}
\end{equation} 
The first equation in this system should be understood as taking the long time limit of the first equation in the parabolic-parabolic system \eqref{eq:PKSeps}. In particular, the initial condition for $\rho^\eps$ is no longer relevant, but we need to add the positivity and mass conservation constraints, which are the final two conditions in the above  elliptic-parabolic system.
Indeed, the first equation in \eqref{eq:PKSeps} is a nonlinear diffusion equation with drift,
and when the drift term is a regular enough fixed potential $\phi$, it has been shown that the corresponding weak solution 
converges uniformly in space as $t \to +\infty$ to weak solutions of the corresponding quasilinear elliptic equation \cite[Theorem 5.1]{BH86}:
\begin{equation}\label{eq:rhom}
	\left\{
	\begin{aligned}
		- \pdiv{\rho \nabla f'(\rho)} + \div(\rho \na \phi ) = 0 &\qquad \text{in } \Omega, \\
		(-\rho\nabla f'(\rho) + \rho\nabla\phi) \cdot n = 0 &\qquad \text{on } \partial\Omega, \\
		\rho \ge 0 \qquad \text{and} \qquad &\int_{\Omega} \rho\, dx=1.
	\end{aligned}
	\right.
\end{equation}

However, it is also known that solutions of this limiting elliptic equation are typically not unique unless further restrictions on $\phi$ are imposed, e.g., strict concavity. Such restrictions are unreasonable for the systems we have in mind where $\phi$ describes the concentration of some chemoattractant. 

In view of this ill-posedness, we observe that the elliptic equation \eqref{eq:rhom} is the Euler-Lagrange equation for the following strictly convex minimization problem where $\phi \in H^1(\Omega)$ is prescribed:
\begin{equation}\label{eq:minrhom}
	\rho_\phi :=\argmin_{\rho \in  L^m(\Omega)} \mathscr{K}(\rho;\phi)	\quad\text{where}\quad
	\sK(\rho;\phi) :=
	\begin{cases} 
		\int_\Omega f(\rho) - \rho\phi \,dx  & \text{if } \rho \in \cP_\ac(\Omega) \cap L^m(\Omega), \\
		+\infty & \text{if } \rho \in L^m(\Omega) \setminus \cP_\ac(\Omega).
	\end{cases}
\end{equation}
We denote by $\cP_\ac(\Omega)$ the collection of probability measures on $\Omega$ that are absolutely continuous with respect to the $d$-dimensional Lebesgue measure on $\Omega$, and we always identify such a measure with its density. We thus choose this minimization problem's unique minimizer, denoted by $\rho_\phi$, to be the weak solution to our elliptic equation \eqref{eq:rhom}. 

While this choice may seem arbitrary, we point out that 
this minimizer is stable with respect to a large class of vanishing diffusions (linear or nonlinear).
In particular, the inclusion of a linear diffusive term $-\nu\lap\rho$ to \eqref{eq:rhom} (modeling some natural noise in the motion of the organisms) leads to a perturbed functional $\sK_\nu(\rho) :=\sK(\rho) + \nu \int_\Omega\rho \ln \rho \,dx $
and one can show that the minimizer $\rho_{\phi,\nu}$ of $\sK_\nu$,
which is the unique solution of the corresponding elliptic equation  \eqref{eq:rhom},
 converges to that of $\sK$ when $\nu\to0$ (in the sense of the narrow convergence of measure). 
In anticipation of the discussion in Section \ref{sec:results}, we also note that this choice of weak solution for $\rho$ is natural in view of the 
 assumption of convergence of the energy \eqref{eq:EA} that we will make to study  the limit $\eps\to0^+$.
\medskip

The existence and properties of the solution of the minimization problem \eqref{eq:minrhom} will be studied in Proposition \ref{prop:rho}. With this unambiguous definition of $\rho_\phi$,  we can now rewrite \eqref{eq:PKS0} as a parabolic equation 
\begin{equation} \tag{$P_\varepsilon$} \label{eq:phi}
	\left\{
	\begin{aligned}
		&\pa_t \phi^\eps -  \Delta \phi^\eps = \eps^{-2}\left( \rho_{\phi^\eps}-\sigma \phi^\eps \right) &&  \mbox{in }  (0,\infty)\times \Omega, \\
		&\nabla\phi^\varepsilon\cdot n = 0	&& \text{on } (0,\infty)\times\partial\Omega, \\
		&\phi^\varepsilon(0,\cdot) = \phi_\init^\varepsilon && \text{in } \Omega.
	\end{aligned}
	\right.
\end{equation}
with a nonlinear, nonlocal operator $\phi(t,\cdot)\mapsto \rho_{\phi(t,\cdot)}$ that maps the chemoattractant concentration profile at time $t$, i.e., $\phi(t,\cdot)$, to the solution $\rho_{\phi(t,\cdot)}$ of the minimization problem \eqref{eq:minrhom}.

\medskip

The asymptotic behavior of $\phi^\eps$ when $\eps\to 0^+$ is obviously strongly dependent on this nonlocal operator, which in turn depends on the choice of nonlinearity $\rho\mapsto f(\rho)$.
While we will list the necessary assumptions on $f$ in the next section, the typical 
nonlinearity for which our results hold is the aforementioned power law $f_m \colon \bR \to \bR \cup \{+\infty\}$
\begin{equation} \label{eq:fm}
	f_m(u):=\begin{cases}
		\frac{u^m}{m-1}	&	\text{if } u \ge 0, \\
		+\infty			&	\text{otherwise.}
	\end{cases} 
\end{equation}
with $m>2$. In that case, Proposition \ref{prop:rho} gives the explicit formula
$$ \rho_\phi(x) = c_m (\phi(x)-\ell(\phi))_+^{\frac 1 {m-1}} , \qquad \int_\Omega \rho_\phi(x)\, dx=1 \qquad \mbox{ with } c_m :=\left(\frac{m-1}{m}\right)^{\frac{1}{m-1}}$$
where $\ell(\phi)$ is a real number (the Lagrange multiplier) determined by the mass constraint on the organisms. The   parabolic equation  \eqref{eq:phi} then takes the form:
\begin{equation}\label{eq:ACe}
\pa_t\phi^\eps - \Delta \phi^\eps = \eps^{-2} \left( c_m (\phi^\eps-\ell(\phi^\eps))_+^{\frac 1 {m-1}} - \sigma\phi^\eps \right)
\end{equation}
and  we observe that for fixed $\ell$ (small enough), the function $v\mapsto c_m (v-\ell)_+^{\frac 1 {m-1}} - \sigma v$
is a  bistable nonlinearity. Equation \eqref{eq:ACe} can then be seen as a type of Allen-Cahn equation, and the 
 limit $\eps\to0^+$ of such an equation classically leads to mean-curvature flows.

	Of course, the fact that $\ell$ depends on $\phi$ in \eqref{eq:ACe} makes our equation highly nonlocal and means that new arguments will be needed to justify the limit. 
But  on the flip side, this dependence   comes from the mass constraint $\int_\Omega \rho_{\phi}(t,x)\, dx =1$ and 
this mass constraint implies a mass constraint on $\phi^\eps$ in the limit $\eps\to0^+$. Indeed, integrating \eqref{eq:phi} with respect to $x$ shows that the mass $m^\eps(t) = \int_\Omega \phi^\eps(t)\, dx$ solves
$\frac{d}{dt} m^\eps(t)  = \eps^{-2} (1-\sigma m^\eps(t))$ 
and  so
	\begin{equation}\label{eq:L1} 
		\int_\Omega \phi^\eps(t,x)\, dx \to \frac 1 \sigma 
		\qquad  \mbox{ as $\eps\to 0$ for all $t>0$}
	\end{equation}
(independently of the initial mass $\int \phi_{in}(x)\, dx$).
We thus see that while the equation \eqref{eq:phi} does not preserve the mass: The limiting flow should  be volume-preserving. 
	\medskip
	
	The main result of this paper (Theorem \ref{thm:main}) makes this heuristic rigorous and shows that the solution of \eqref{eq:phi} converges, when $\eps\to0^+$ to the characteristic function of a set whose evolution is described by the following \emph{volume-preserving mean-curvature flow}:
	\begin{equation}\tag{$P_0$}\label{eq:V} 
		 V = -  \kappa + \Lambda\qquad \mbox{ on } \pa E(t)\cap \Omega, \qquad |E(t)| = \frac 1 \sigma \qquad	\forall \, t > 0,
	\end{equation}
	where $V = V(t,x)$ is the normal velocity of a point $x$ on  the boundary $\partial E(t)$ of some set $E(t) \subset \cl \Omega$;  
	$\kappa(t,x)$ denotes the mean-curvature of $\pa E(t)$ at $x \in \pa E(t)$ (with the convention that $\kappa(t,x)>0$ at points $x \in \partial E(t)$ for which $\partial E(t)$ is locally convex); and $\Lambda=\Lambda(t)$ is a Lagrange multiplier that enforces the volume constraint. Further, the corresponding distribution of organisms $\rho_{\phi^\varepsilon}$ converges to a scalar multiple of the limit of $\phi^\varepsilon$, so this resolves the asymptotic dynamics we ultimately set-out to describe.

	\medskip

	We will prove our main result when $f=f_m$ with $m\in(2,\infty)$ and for more general nonlinearities that have similar properties.
As a final comment, we note that several recent works (e.g., \cite{KMW,KMW2}) have also focused on the related \emph{congested} (or incompressible, or hard-sphere) model, which corresponds to the nonlinearity 
$$
f_\infty(u) :=
\begin{cases}
	0 & \mbox{ if } 0 \le u\leq 1,\\
	+\infty & \text{ otherwise.}
\end{cases}
$$
While we will not investigate the corresponding limit in this paper, a similar convergence result is expected to hold in that framework as well.
One issue in that case is that, depending on $\phi$, the minimization problem \eqref{eq:minrhom} may not have a unique solution in that case (the corresponding energy is not strictly convex). This may be remedied by adding a small linear diffusion term that goes to zero with $\eps$, but we will not pursue this analysis in this paper.

\subsection{Outline of the paper}
The next section presents our main results in detail: The well-posedness of \eqref{eq:phi} (Theorem \ref{thm:existence}) and the convergence to the volume-preserving mean curvature flow (Theorem \ref{thm:main}).
Section \ref{sec:rho} is devoted to the minimization problem \eqref{eq:minrhom} which defines $\rho$ in terms of $\phi$ (Proposition \ref{prop:rho}).		
In Section \ref{sec:exist} we prove Theorem \ref{thm:existence} (existence of $\phi^\eps$ for $\eps>0$).
The rest of the paper is devoted to the limit $\eps\to0^+$:
First (Section \ref{sec:Gamma}) we prove the $\Gamma$-convergence of the energy (Theorem \ref{thm:Gamma}). Then 
in Section \ref{sec:conv} we prove the first part of Theorem \ref{thm:main} (phase separation, that is convergence of $\phi^\eps$ and $\rho_{\phi^\varepsilon}$ characteristic functions).
The derivation of the volume-preserving mean-curvature flow is detailed in Sections \ref{sec:V} and \ref{sec:mc} 
		
		\medskip


\section{Main results}\label{sec:results}
We now make precise the assumptions on the nonlinearity $f$ that are used in our analysis. This first collection of hypotheses is  used to prove  the well-posedness of \eqref{eq:phi}. Additional hypotheses needed to show  phase separation and derive the volume-preserving mean-curvature flow will be  described later.
\smallskip

\begin{wellPosednessHypotheses-f}
	\item \label{hyp:fMinimalAssumptions} $f \in C^1([0,\infty)) \cap C^2((0,\infty))$,  $f(0) = f'(0) = 0$, and $f$ is strictly convex on $[0,+\infty)$;
	\medskip
	
	\item \label{hyp:fGrowthCondition} \setcounter{growthCondition-f}{\value{wellPosednessHypotheses-fi}} There exists $m > 2$, $R_1 > 0$, and $C_1 > 0$ such that $f(u) \ge C_1 u^m$ whenever $u > R_1$.
\end{wellPosednessHypotheses-f}
\medskip

These assumptions are natural from the point of view of the variational methods we aim to use. Indeed, \ref{hyp:fGrowthCondition} implies $f$ is superlinear $\lim_{u \to +\infty} f(u)/u = +\infty$, and superlinearity and convexity of $f$ are necessary to ensure the lower semicontinuity w.r.t.\ narrow convergence of the entropy-like term $\rho \mapsto \int_\Omega f(\rho) \,dx$ in $\sK$ (see \cite[Proposition 7.7]{OTAM}). In view of the mass constraint on $\rho$, the function $f$ is defined up to an affine function, since the addition of any such function to $f$ only changes the energy $\sK$ by a constant. Thus, the assumption $f(0) = f'(0) = 0$ is no restriction. \emph{Strict} convexity \ref{hyp:fMinimalAssumptions} will imply that $\sK$ has a unique minimizer whenever $\phi$ is prescribed, so the map $\phi \mapsto \rho_\phi$ is uniquely defined.

On the one hand, these two assumptions will be enough to prove the existence of a solution to \eqref{eq:phi} and will be essential in the limit $\eps\to 0^+$. 
In particular, the restriction $m>2$ in \ref{hyp:fGrowthCondition} plays a key role in this limit (as explain in \cite{KMJW}). 
We note that the existence of a solution can be established when $m=2$, but this requires additional work  since the coercivity estimate 
\eqref{eq:coerc} does not hold in that case.
On the other hand, these two conditions do not appear to be sufficient to prove the uniqueness of a solution to \eqref{eq:phi}, which we will establish under the stronger condition:
\medskip

\begin{wellPosednessHypotheses-f}[start=\value{growthCondition-f}+1]
\item \label{hyp:fStronglyConvex}  	(uniform convexity)  There exist $\alpha>0$ such that $\inf_{u > 0} f''(u) \ge \alpha$. 
\end{wellPosednessHypotheses-f}
\medskip

This condition implies that the mapping $\phi \to \rho_\phi$ is Lipschitz $L^2(\Omega)\mapsto L^2(\Omega)$, which is sufficient to obtain uniqueness of our weak solution to the overall system. 
The typical nonlinearities $f$ for which   \eqref{eq:phi} has a solution and our convergence result  (Theorem \ref{thm:main} below) holds are power law nonlinearities given by:
	\begin{equation} \label{eq:porousMediumf}
		f(u) :=
		\begin{cases}
			\displaystyle\frac{u^m}{m-1}	& \text{if } u \ge 0, \\
			+\infty 		& \text{otherwise.}
		\end{cases}
		\quad m>2.
	\end{equation}
Such a function satisfies \ref{hyp:fMinimalAssumptions} and \ref{hyp:fGrowthCondition}, but not \ref{hyp:fStronglyConvex}.
An example of nonlinearity that satisfies all three assumptions is:
	\begin{equation} \label{eq:porousMedium-stronglyConvexf}
		f(u) :=
		\begin{cases}
			\displaystyle\frac{u^m}{m-1} + \frac{\alpha}{\beta(\beta-1)}u^\beta	& \text{if } u \ge 0, \\
			+\infty 		& \text{otherwise}
		\end{cases}		
	\end{equation}
	with $m >2$, $\alpha > 0$ and $\beta \in (1,2]$. 
It would also be possible to replace $\frac{u^\beta}{\beta(\beta-1)}$ by $u\ln u$ (which corresponds to the inclusion of $-\alpha\lap \rho$ in \eqref{eq:rhom}) although \ref{hyp:fMinimalAssumptions} would have to be adjusted in such a case. 
In particular, the limit $\eps\to 0^+$ would be significantly different, since the density and chemoattractant separate into \emph{two positive phases}, i.e., two regions of differing positive densities, instead of one. The resulting volume-preserving mean-curvature flow is thus two-phase instead of one-phase.

\subsection{Well-posedness of  \eqref{eq:phi}}
Our first task is to prove the existence and uniqueness of $\rho_\phi$, the solution of the minimization problem \eqref{eq:minrhom} for a given $\phi \in H^1(\Omega)$. 

We first prove:
\begin{proposition}[Well-posedness of \eqref{eq:rhom}]\label{prop:rho}
	Suppose $f$ satisfies \ref{hyp:fMinimalAssumptions} and \ref{hyp:fGrowthCondition} (it is enough to assume that $m\in[2,+\infty)$).
	\begin{enumerate}[label=(\roman*)] 
		\item (existence and uniqueness) Given $\phi\in L^2(\Omega)$, there exists a unique $\rho_\phi \in \cP_\ac(\Omega) \cap L^m(\Omega)$ that solves the minimization problem \eqref{eq:minrhom}. Furthermore, we have
\begin{equation}\label{eq:rhophibd} 
\| \rho\|_{L^m(\Omega)} ^m \leq C+C\|\phi\|_{L^2(\Omega)}^2
\end{equation}
for some constant $C>0$ depending only on $f$ and $\Omega$.		
		\item (support of the minimizer) Under the same conditions, there exists $\ell \in \bR$ (called \emph{the Lagrange multiplier for the mass constraint}) such that
		\begin{equation} \label{eq:supportOfRho}
			f'(\rho_\phi) - \phi = -\ell	\qquad \rho_\phi\text{-a.e.}
		\end{equation}
		
		\item (Euler-Lagrange equation) 
		If $\phi \in H^1(\Omega)$, then $f'(\rho_\phi) \in H^1(\Omega)$ and
$ - \rho_\phi \na f'(\rho_\phi) + \rho_\phi \na \phi = 0$ a.e.\  in $  \Omega$.
In particular,  $\rho_\phi$ is a weak solution of the quasilinear elliptic equation \eqref{eq:rhom}; that is,
		\begin{equation}\label{eq:rhoWeakSolution}
			\int_\Omega \rho_\phi \nabla f'(\rho_\phi) \cdot \nabla \zeta - \rho_\phi \nabla\phi \cdot \nabla\zeta \,dx = 0 	\qquad	\forall \, \zeta \in C^\infty(\cl\Omega).
		\end{equation}
		
	\end{enumerate}	
\end{proposition}

Next, we define the energy functional associated to \eqref{eq:phi}: 
\begin{align*}
	\sE_\eps (\phi)
	&: = \inf_{\rho \in L^m(\Omega) } \left\{ \frac 1 \eps \int_\Omega f(\rho) -  \rho \phi + \frac \sigma 2   |\phi|^2 \, dx + \frac{\eps} 2 \int_\Omega  |\na \phi|^2\, dx\right\}  \\
	& =  \frac 1 \eps \int_\Omega f(\rho_\phi) -  \rho_\phi \phi + \frac\sigma 2   |\phi|^2 \, dx +\frac {\eps} 2  \int_\Omega  |\na \phi|^2\, dx,
\end{align*}
where $\rho_\phi$ is the unique solution to \eqref{eq:minrhom} for a given $\phi\in H^1(\Omega)$. A formal computation shows that smooth solutions of \eqref{eq:phi} satisfy the energy dissipation inequality
\begin{equation}\label{eq:energy0}
	\sE_\eps(\phi^\eps(t)) +   \int_0^t  \int_\Omega \eps |\pa_t \phi^\eps(t,x)|^2\, dx \, dt \leq \sE_\eps(\phi_{\init}).
\end{equation}
In fact, \eqref{eq:phi} can be interpreted as a gradient flow for this energy functional with respect to the $L^2(\Omega)$-inner product, and we exploit this observation in order to prove the following result:
\begin{theorem}[Well-posedness of \eqref{eq:phi}]\label{thm:existence}
	Suppose $f$ satisfies \ref{hyp:fMinimalAssumptions} and  \ref{hyp:fGrowthCondition}. For any $T,\varepsilon,\sigma>0$ and nonnegative $\phi_{\init}\in H^1(\Omega)$, 
there exists a nonnegative strong solution $\phi^\eps$ of \eqref{eq:phi} such that		
	\begin{equation}\label{eq:phispace}
	\phi^\eps\in H^1(0,T; L^2(\Omega)) \cap L^\infty(0,T;H^1(\Omega)) \cap   L^2(0,T;H^2(\Omega))
	\end{equation}
and the energy dissipation inequality \eqref{eq:energy0} holds. 	
	If additionally $f$ satisfies \ref{hyp:fStronglyConvex}, then this solution is unique.
	
	Finally, if $f$ satisfies $f'(u) \ge u/C_2$ whenever $u> R_2$ and $\phi_\init \in L^\infty(\Omega)$, 
then $\phi^\eps$ and $\rho_{\phi^\eps}$ belong to $L^\infty(0,T; L^\infty(\Omega))$ and 
we have the estimate
	\begin{equation}\label{eq:linftyEstimate}
		\| \phi^\eps(t) \|_{L^\infty(\Omega)} \le C( \norm{\phi_\init}_{L^\infty(\Omega)}, C_2, R_2, |\Omega|,f)  \exp\left( \frac{C_2 - \sigma}{\varepsilon^2} t \right) \qquad \forall t\geq 0.
	\end{equation}
\end{theorem}
Since \eqref{eq:phispace} implies $\phi^\eps\in C^{0,\frac{1}{2}}([0,T]; L^2(\Omega))$, the initial condition is satisfied 
in the classical sense. This also shows that the function $t \mapsto \rho_{\phi^\eps(t)}$ is well-defined (with  $\rho_{\phi^\eps(t)}$  the unique minimizer of \eqref{eq:minrhom} with $\phi  = \phi^\eps(t)\in L^2(\Omega)$). Finally, we note that if $C_2 \le \sigma$, then \eqref{eq:linftyEstimate}  is uniform-in-$\varepsilon$.
This is the case when $f$ is given by $f_m$ with $m>2$ ($C_2$ can be chosen arbitrarily small in that case), which is consistent with the convergence result of the next section ($\phi^\eps$ converges to a multiple of a characteristic function as $\eps\to0^+$).

\subsection{Singular limit and derivation of volume-preserving mean-curvature flow}
The second part of the paper concerns the limit $\eps\to 0^+$ of  \eqref{eq:phi}. 
A key role in this limit is played by the energy $\sE_\eps$ dissipated by $\phi^\eps$,
which can be rewritten as 
\begin{equation}\label{eq:energyE_eps}
	\sE_\eps (\phi) =  \inf_{\rho\in L^m(\Omega)} \left\{ \frac 1 \eps \int_\Omega f(\rho)  - \frac  1 {2\sigma} \rho^2  \, dx +  \frac 1 \eps \int_\Omega\frac 1 {2\sigma}  (\rho-\sigma \phi )^2\, dx +\eps  \int_\Omega \frac1  2 |\na \phi |^2\, dx \right\}.
\end{equation}
Since $\int_\Omega \rho=1$, we can augment this energy with an additional constant $\frac{a}{\varepsilon}\int_\Omega \rho$ (for a constant $a \in \bR$ to be chosen later), and it will remain dissipated by the solution. We thus define:
\begin{align}
	\sJ_\eps(\phi) 
	& :=\inf_{\rho\in L^m(\Omega)} \left\{ \frac 1 \eps \int_\Omega f(\rho) + a\rho - \frac 1 {2\sigma} \rho^2  \, dx +  \frac 1 \eps \int_\Omega\frac 1 {2\sigma}  (\rho-\sigma \phi )^2\, dx +\eps  \int_\Omega \frac1  2 |\na \phi |^2\, dx \right\} \nonumber \\
	& = \frac 1 \eps \inf_{\rho\in L^m(\Omega)} \left\{  \int_\Omega W(\rho) + \frac 1 {2\sigma}  (\rho-\sigma \phi )^2\, dx \right\}+\eps  \int_\Omega \frac 1 2 |\na \phi |^2\, dx. \label{eq:J0}
\end{align}
Above, we have introduced
\begin{equation}\label{eq:doubleWellPotential-Density}
	W(u):=f(u) + au - \frac 1 {2\sigma} u^2.
\end{equation}
When $f$ satisfies \ref{hyp:fGrowthCondition} (more than quadratic growth), this function is bounded below and in the particular case  where $f = f_m$ is the power law given by \eqref{eq:fm} with $m \in (2,\infty)$, it is not difficult to check that $W$ is a double-well potential with a singular well at $u=0$ and a smooth well at $u  = \theta_m :=\left(\frac 1{2\sigma}\right)^{\frac{1}{m-2}}$ (see Figure \ref{fig:1}.
\begin{figure}[]
	\includegraphics[width=.35\textwidth]{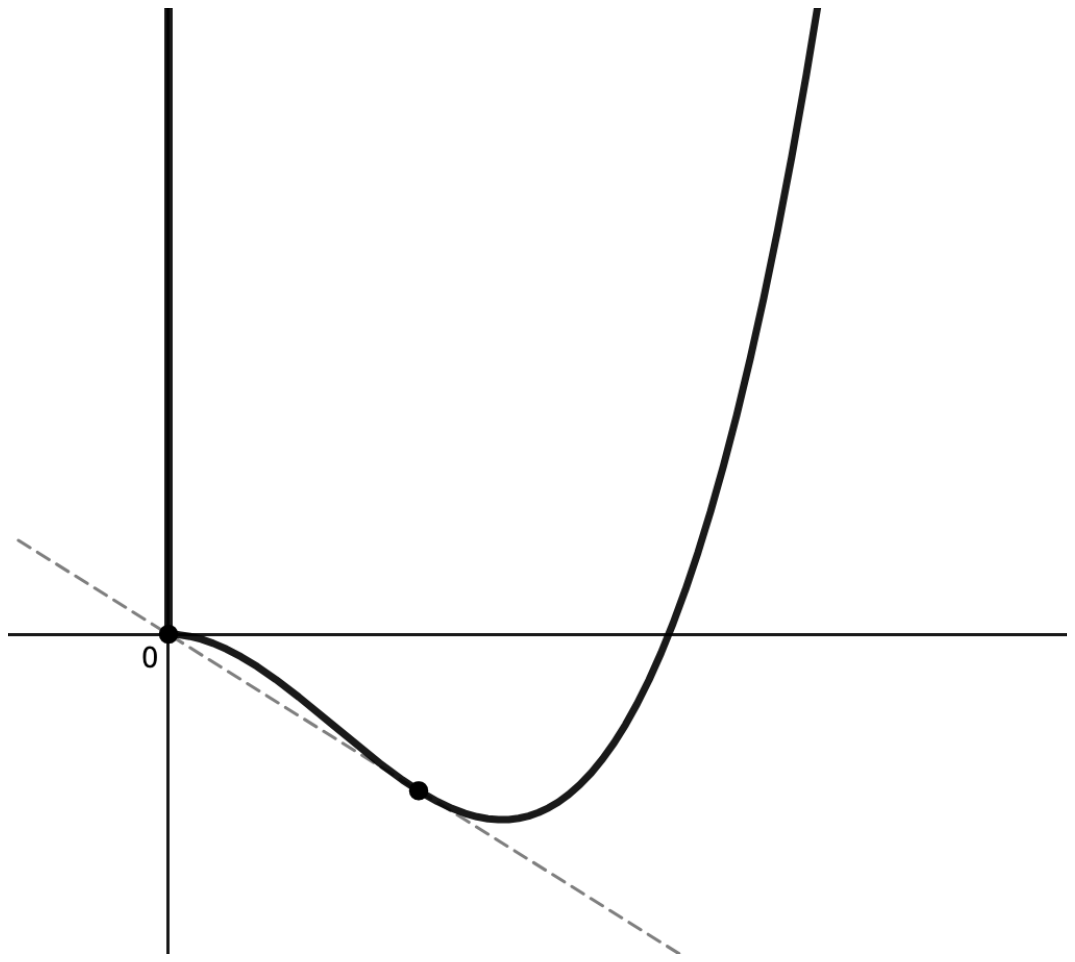}
 	 		\caption{The double-well potential $f(u) - \frac{1}{2\sigma} u^2$ when $f(u)=\frac 1 2 u^3$}
			 	 		\label{fig:1}
\end{figure}

If we choose $a=a_m:=\frac{m-2}{m-1} \theta_m^{m-1}$, we have $W(u)\geq 0$ for all $u\geq 0$ with equality only at $u=0$ and $u=\theta_m$. 
The addition of this constant $\frac{a}{\varepsilon}\int\rho\, dx$ to the energy, while immaterial for the dissipation, is thus essential for the limit $\varepsilon \to 0^+$, as it ensures that the energy has a finite limit when $\varepsilon \to 0^+$. 

\medskip

This double-well property of $W$ is the key feature that will lead to phase separation and mean-curvature phenomena. 
We do not need to restrict ourselves to power law nonlinearities, but we will make the following assumption in this section:
\medskip

\begin{singularLimitHypotheses-f}[start=4]
	\item \label{hyp:fSecondBounds} The function $u\mapsto W(u)$ defined by \eqref{eq:doubleWellPotential-Density}  is a double-well potential with wells at $u=0$ and $u=\theta>0$. The constant $a\in \bR$ in \eqref{eq:doubleWellPotential-Density} is then chosen so that
	\begin{equation}\label{eq:dw}
	W(u) \geq 0, \; \forall u\geq 0 \quad\mbox{ and } \quad  \{ W = 0 \} = \{0,\theta\} \mbox{ for some $\theta>0$}.
\end{equation}
\end{singularLimitHypotheses-f}
\medskip

We refer to \cite[Section 2.1]{M23} for assumptions on $f$ that ensure that \ref{hyp:fSecondBounds} holds
and we note that the key idea is that we should have 
$f''(u) \le \frac{1}{\sigma}$ for small $u$ and $f''(u) \ge \frac{1}{\sigma}$ for large $u$.
Heuristically, it is the fact that the convexity of  $ f(u)  - \frac 1 {2\sigma} u^2$ changes that leads to phase separation phenomena.

\medskip

Next, we introduce the function\footnote{When $W$ is a convex function, this function $W_\sigma$ is known as the \emph{Moreau-Yosida regularization of $W$}.}
\begin{equation}\label{eq:g}
	W_\sigma (v) : = \inf_{u\geq 0}  \left[ W(u) +  \frac {1}{ 2\sigma} (u- v)^2\right]. 
\end{equation}
We then see that
\begin{equation}\label{eq:JW}
	\sJ_\eps(\phi) \geq \frac 1 \eps   \int_\Omega  W_\sigma (\sigma \phi ) \, dx  +\eps  \int_\Omega \frac 1 2 |\na \phi |^2\, dx.
\end{equation}
We note that \ref{hyp:fSecondBounds}  implies that $W_\sigma(v)\geq 0$ for all $v\geq 0$ and it is easy to check that $W_\sigma(0)=W_\sigma(\theta) =0$.
In fact, $W_\sigma$ is also a double-well potential (see Lemma \ref{lem:gF1}), so that the functional on the right-hand side is a Modica-Mortola functional that, as $\varepsilon \to 0^+$, $\Gamma$-converges w.r.t.\ $L^1(\Omega)$ to (a multiple of) the perimeter functional.
We prove a corresponding result for $\sJ_\eps$. 

\begin{theorem}[$\Gamma$-convergence of \eqref{eq:J0}]\label{thm:Gamma}
	Assume that $\Omega$ is an open subset of $\bR^d$ with Lipschitz boundary and that $f$ satisfies \ref{hyp:fMinimalAssumptions}, \ref{hyp:fGrowthCondition} and \ref{hyp:fSecondBounds}.
	Then as $\varepsilon \to 0^+$
	the functional  $\sJ_\eps$ defined by \eqref{eq:J0} $\Gamma$-converges to $\sJ_0$ defined by
	$$
	\sJ_0(\phi):=\begin{cases}
		\displaystyle \gamma  \int_\Omega |\na \phi|   & \mbox{ if } \phi \in BV(\Omega;\{0,\theta/\sigma\}) \mbox{ and } \int_\Omega \phi\, dx =1/\sigma, \\
		+\infty & \mbox{ otherwise in $L^1(\Omega)$}
	\end{cases}
	$$
	with
	\begin{equation}\label{eq:sigma}
		\gamma:=\frac 1 {\theta}\int_0^{\theta } \sqrt{2 W_\sigma (v)}\, dv.
	\end{equation}
	More precisely, we have:
	\begin{enumerate}[label=(\roman*)] 
		\item {$\bm{\liminf}$ \textbf{property}}: For all $\phi \in L^1(\Omega)$, positive sequences $\varepsilon_n \to 0$, and $\pseq{\phi^{\varepsilon_n}}_n \subset L^1(\Omega)$ such that $\phi^{\eps_n} \to \phi$ in $L^1(\Omega)$, we have
		$$\liminf_{n\to \infty } \sJ_{\eps_n}(\phi^{\eps_n}) \geq \sJ_0(\phi).$$
		\item {$\bm{\limsup}$ \textbf{property}}: For all $\phi \in L^1(\Omega)$ and positive sequences $\eps_n \to 0$, there exists a sequence $\pseq{\phi^{\eps_n}}_n\subset L^1(\Omega)$ such that $\phi^{\eps_n}\to\phi$ in  $L^1(\Omega)$ and 
		$$\limsup_{n\to \infty}  \sJ_{\eps_n}(\phi^{\eps_n}) \leq \sJ_0(\phi).$$
	\end{enumerate}
\end{theorem}
Note that $\sJ_0(\phi)<+\infty$ if and only if $\phi= \frac{\theta}{\sigma} \chi_E$ for some set $E\subset\Omega$ with finite perimeter $P(E,\Omega)<+\infty$ and with volume $|E|=\frac{1}{\theta}$.
The existence of such a function requires $\frac{1}{\theta} \leq |\Omega|$; heuristically, the domain needs to be sufficiently large to accommodate a given amount of mass of a given density. When $\frac{1}{\theta} = |\Omega|$, the only possibility is (up to sets of Lebesgue measure zero) $E=\Omega$. In such a case, the limiting volume-preserving mean-curvature flow is stationary, so in order for anything interesting to happen, we will  always assume that
\begin{equation}\label{eq:st}
	\theta |\Omega|>1.
\end{equation}

We stress again the fact that even though $\sJ_\eps$ does not include a mass constraint on $\phi$, its $\Gamma$-limit $\sJ_0$ does. 
As we will see in the proof, this is due to the fact that the mass constraint on $\rho$ is transferred to $\phi$ in the limit $\eps\to 0^+$.
The fact that the solutions of \eqref{eq:phi} satisfy a mass constraint in the limit was already seen in \eqref{eq:L1}, but this 
 property of $\sJ_0$ is  more rigorous evidence that the evolution of the limit ought to be a volume-preserving mean-curvature flow, 
 which is the content of our main theorem:

\begin{theorem}[Phase separation \& conditional convergence of \eqref{eq:phi} to \eqref{eq:V}]\label{thm:main}
	Suppose $f$ satisfies \ref{hyp:fMinimalAssumptions}, \ref{hyp:fGrowthCondition}  and \ref{hyp:fSecondBounds}. Fix any positive sequence $\pseq{\eps_n}_n$ such that $\eps_n \to 0$. Let $\pseq{\phi_{\init}^{\eps_n}}_n \subset H^1(\Omega)$ be such that $\sup_n\sJ_{\eps_n}(\phi_{\init}^{\eps_n})\leq C$ (for some constant $C > 0$), and let  $\phi^{\eps_n}$ be a solution of \eqref{eq:phi} with initial data $\phi_{\init}^{\eps_n}$ (provided by Theorem \ref{thm:existence}).
	\begin{enumerate}[label=(\roman*)]	
		\item (phase separation) There exists a subsequence, still denoted by $\pseq{\eps_n}_n$, such that $\pseq{\phi^{\eps_n}}_n$ converges strongly in $C^{0,s}([0,T]; L^1(\Omega))$ for all $s\in (0,1/2)$ to a function
		\begin{equation}
			\phi^0\in C^{0,s}([0,T]; L^1(\Omega)) \cap  L^\infty\big(0,T;BV(\Omega;\{0,\theta/\sigma\})\big) \cap BV\big((0,T)\times\Omega\big).
		\end{equation}
		Along this subsequence, $\pseq{\rho^{\eps_n}}_n$ converges  to  $\sigma \phi^0$
		strongly in $L^\infty([0,T]; L^q(\Omega))$ for all $q \in [1,\frac{d}{d-1})$.
		\smallskip
		
		\item (conditional convergence) Assume furthermore that the following energy convergence  condition holds:
		\begin{equation}\label{eq:EA}
			\lim_{n\to \infty} \int_0^T \sJ_{\eps_n}(\phi^{\eps_n}(t)) \, dt = \int_0^T \sJ_0(\phi^0(t))\, dt.
		\end{equation}
		Then there exists $V\in L^2((0,T)\times\Omega;\abs{\na \phi^0}dt)$ such that 
		$ \pa_t\phi^0 = V |\na \phi^0|dt$ and  $\phi^0(0,\cdot)=\phi_{\init}$ in the following distributional sense:
		\begin{equation}\label{eqLvelweak}
			\int_0^T \int_\Omega \phi^0 \pa_t \zeta \, dx\, dt+ \int_\Omega \phi_{\init} \zeta(0)\, dx = -\int_0^T\int_\Omega V \zeta |\na \phi^0|dt
		\end{equation}
		for all $\zeta \in C^1_c([0,T)\times\Omega)$. Moreover,  
		there exists $\Lambda\in L^2(0,T)$ (the \emph{Lagrange multiplier for the volume constraint}) such that \eqref{eq:V} holds in the following distributional sense:
		\begin{equation}\label{eq:Vxi} 
			\int_0^T\int_\Omega 
			V \xi \cdot \na \phi^0 \, dt =  \int_0^T\int_\Omega\left( \div \xi - \frac{\na \phi^0}{|\na \phi^0|} \otimes \frac{\na \phi^0}{|\na \phi^0|}: D\xi \right) |\na \phi^0|\, dt + \int_0^T \Lambda(t) \int_\Omega  \xi \cdot \na \phi^0\, dt
		\end{equation}
		for all vector fields $\xi \in [C^1([0,T]\times\overline \Omega)]^d $ satisfying $\xi \cdot n=0$ on $\pa\Omega$.
	\end{enumerate}
\end{theorem}

Since (i) implies that $\phi^0(t,x)=\frac \theta \sigma \chi_{E(t)}(x)$ for some set $E(t) \subset \Omega$ with finite perimeter in $\Omega$, \eqref{eqLvelweak} encodes the fact that $\pa E(t)$ moves with a velocity whose normal component is $V$
and \eqref{eq:Vxi} encodes the velocity law $ V = -  \kappa + \Lambda$: We recognize the first term on the right-hand side of \eqref{eq:Vxi} as the distributional mean-curvature of the reduced boundary $\partial^\ast E(t)$ (see \cite[Section 17.3]{2012_Maggi_setsOfFinitePerimeter} or \cite[Section 7.3]{AFP_BV_00}). The object $\nabla \phi^0(t)/\abs{\nabla \phi^0(t)}$ denotes the Radon-Nikodym derivative of the Radon measure $\nabla \phi^0(t)$ w.r.t.\ its total variation $\abs{\nabla \phi^0(t)}$; we recall that it coincides with the measure-theoretic inner unit normal vector to $E(t)$, the support of $\phi^0(t)$.

We stress the fact that part (ii) of the theorem (the derivation of the volume-preserving mean-curvature flow \eqref{eq:V}) is established under the \emph{additional} assumption of convergence of the energy \eqref{eq:EA}. 
The $\liminf$ inequality of the $\Gamma$-convergence result (Theorem \ref{thm:Gamma}-(i))
gives 
\[
\liminf_{n\to \infty} \int_0^T \sJ_{\eps_n}(\phi^{\eps_n}(t)) \, dt \geq \int_0^T \sJ_0(\phi^0(t))\, dt,
\]
so this additional assumption \eqref{eq:EA} states that a sequence of solutions $\pseq{\phi^{\varepsilon_n}}_n$ to \eqref{eq:phi} loses no energy in the limit $\varepsilon_n \to 0$. Such energy losses are typically caused by folding or disappearing boundaries.
This type of assumption has been used to obtain convergence results in several related frameworks (see for example Luckhaus and Sturzenhecker \cite{1995_Luckhaus&Sturzenhecker_implicitTimeDiscretization-mcf}, Laux and Otto \cite{LO16}, Laux and Simon \cite{Laux3}).

Let us also point out that by making the assumption that the initial data is well-prepared, i.e., $\sup_{n\in\bN} \sJ_{\varepsilon_n}(\phi_\init^{\varepsilon_n}) < +\infty$, we are prohibited from fixing some $\phi_\init$ and choosing $\phi_\init^{\varepsilon_n} = \phi_\init$ independent of $\eps_n$. A typical choice of well-prepared initial data is the  recovery sequence of a given $\phi_\init$ for the $\limsup$ inequality of the $\Gamma$-convergence result (Theorem \ref{thm:Gamma}-(ii)).

\subsection{Relationship to prior work}
We end this section with a brief review of previous work on related problems:
First, we recall that the existence of time-global classical solutions to the volume-preserving mean-curvature flow
\eqref{eq:V} was proven by Gage \cite{1986_Gage_areaPreservingEvolution-planeCurves} and Huisken \cite{1987_Huisken_timeglobalClassical-vpmcf} for convex initial data and 
by Escher and Simonett \cite{1998_Escher&Simonett_vpmcf-near-spheres} for 
nonconvex initial data that are sufficiently close to a sphere.
For more general initial data, classical solutions of \eqref{eq:V} can develop self-intersections
and experience curvature blow-up in finite time (see for example \cite{2000_Mayer&Simonett_selfIntersections-vpmcf,2001_Mayer_singularExample_vpmcf,2005_Escher&Ito_dynamicProperties-vpmcf}).
Various notions of weak solutions have been used to get some global-in-time existence result, such as 
 viscosity solutions (see for instance 
 Kim and Kwon \cite{2020_Kim&Kwon_vpmcf-viscositySoln} and Golovaty
 \cite{1997_Golovaty_vpmcf-ReactionDiffusion-viscosity}) 
 and  integral varifold solutions  by Takasao \cite{2017_Takasao_massConvservingAllenCahn2d3d-to-vpmcf-varifolds, 2023_Takasao_massConservingAllenCahn-to-vpmcf-varifolds}.
   In \cite{2016_Mugnai-etal_globalSolutions-vpmcf}, $BV$ solutions to \eqref{eq:V}  are constructed by a minimizing movements scheme, in the spirit of \cite{1995_Luckhaus&Sturzenhecker_implicitTimeDiscretization-mcf}, 
 under an energy convergence assumption analogous to \eqref{eq:EA}. 
A similar result is obtained by Laux and Swartz in \cite{2017_Laux&Swartz_thresholdingSchemes-IncorporatingBulkEffects}  using a thresholding scheme  -- still under an energy convergence assumption.
Weak-strong uniqueness principles for $BV$ solutions of \eqref{eq:V} have also been established by Laux in \cite{2024_Laux_weakStrongUniqueness-vpmcf}.

There is a long history of works related to the sharp interface limit of Allen-Cahn-type equations to mean-curvature flow
including volume-preserving flows.
The mass-conserving  Allen-Cahn equation proposed  by Rubinstein and Sternberg  \cite{Rubinstein-Sternberg92} reads 
\begin{equation}\label{eq:ACc}
\pa_t \phi - \Delta \phi = -\frac 1{\eps^2} W'(\phi)+ \frac{1}{\varepsilon}\lambda(t)
\end{equation}
with the uniform-in-space Lagrange multiplier $\lambda$ enforcing the mass-conservation constraint $\int \phi(t) \,dx = \int \phi(0) \,dx$.
But other models have been proposed with a Lagrange multiplier that is concentrated near the interfacial layer, such as  (see \cite{1997_Golovaty_vpmcf-ReactionDiffusion-viscosity,bretin-brassel-MMAS11}):
\begin{equation}\label{eq:ACn}
\pa_t \phi - \Delta \phi = -\frac 1{\eps^2}W'(\phi) + \frac{1}{\varepsilon}\lambda(t) \sqrt{2W(\phi)}.
\end{equation}

Convergence result for \eqref{eq:ACc} have been established in particular  by Bronsard and Stoth in 
 \cite{1997_Bronsard&Stoth_vpmcf-GinzburgLandau}
for radially symmetric solutions.
Without symmetry  assumptions, Chen, Hilhorst and  Logak \cite{2011_Chen-etal_massConserving-AllenCahn-vpmcf} proved the convergence of solutions of \eqref{eq:ACc}  to classical solutions of \eqref{eq:V} provided such  solutions exist (such results are thus only valid up to the formation of the first singularity).
A similar result is proved by  Alfaro and   Alifrangis \cite{2014_Alfaro&Alifrangis_massConvservingAllenCahn-local&NonlocalLM} for the equation \eqref{eq:ACn}
and more recently, Kroemer and  Laux \cite{2023_Kroemer&Laux_AC-to-VPMCF-relativeEntropy} used a relative entropy method to get a quantitative rate of convergence (still under the assumption that  the limiting problem \eqref{eq:V} has a classical solution).

Concerning the convergence to weak solutions, we note that Takasao \cite{2017_Takasao_massConvservingAllenCahn2d3d-to-vpmcf-varifolds, 2023_Takasao_massConservingAllenCahn-to-vpmcf-varifolds} proved  the convergence toward integral varifold solutions for \eqref{eq:ACn}.
Laux and Simon, \cite{Laux3},  proved the convergence toward $BV$ solutions (analogous to ours)  for a multiphase mass-conserving Allen-Cahn equation under an energy convergence assumption similar to \eqref{eq:EA}. 
This last result is probably the most similar to ours. We note that our equation \eqref{eq:phi} (see also  \eqref{eq:ACe})  does share some similarities with \eqref{eq:ACn} in the sense that the mass constraint is not enforced uniformly across $\Omega$.  However, the right-hand side of \eqref{eq:phi} cannot be separated into a reaction term that drives the separation of phase and an ad-hoc Lagrange multiplier; the two phenomena are taken into account in the definition of $\rho_\phi$. 
We also recall that \eqref{eq:phi} does not actually preserve the mass of $\phi$ as explained in the introduction since the mass constraint is on $\rho$. As far as we know, our work is the first result on such an equation.


\section{The \texorpdfstring{$\rho$}{\textrho} equation (Proof of Proposition \ref{prop:rho})}\label{sec:rho}

We start by proving the existence of a minimizer for  the functional $\rho\mapsto \mathscr{K}(\rho;\phi)$ 
where we recall that 
$$\mathscr{K}(\rho ; \phi)  :=\int_\Omega f(\rho) - \rho \phi\, dx $$
for a given function  $\phi\in L^2(\Omega)$ (see \eqref{eq:minrhom}).

\begin{proof}[Proof of Proposition \ref{prop:rho} (i)]
We note that $\mathscr{K}(\cdot ; \phi) :L^m(\Omega) \cap \cP_\ac(\Omega) \to (-\infty,+\infty]$ and that $\mathscr{K}(\cdot ; \phi)$ is not identically equal to $+\infty$. Indeed, taking $\mu(x) = \frac{1}{|\Omega|}$, we get
\begin{equation}\label{eq:infK}
\inf_{\rho \in \cP_\ac(\Omega) \cap L^m(\Omega)} \mathscr{K}(\rho;\phi) \leq |\Omega| f(|\Omega|^{-1}) + \frac{1}{|\Omega|} \int_\Omega \phi\, dx
\leq  C+ \frac{1}{|\Omega|^{1/2}} \| \phi\|_{L^2(\Omega)}.
\end{equation} 
Furthermore, assumption \ref{hyp:fMinimalAssumptions} implies that $f(u)\geq 0$ for all $u\geq 0$ and together with \ref{hyp:fGrowthCondition} it implies 
$ f(u) \geq C_1 u^m-C_1 R_1^m$ for all $u\geq 0$.
We thus have (recall that $m\geq 2$ and so $\frac{m}{m-1}\leq 2$) for $\rho \in L^m(\Omega) \cap \cP_\ac(\Omega)$:
\begin{align}
\mathscr{K}(\rho ; \phi)  = \int_\Omega f(\rho) - \rho \phi\, dx 
&  \geq C_1 \int_\Omega \rho^m\, dx - \int_\Omega  \rho \phi\, dx - C \nonumber \\
& \geq   \frac{C_1}{2} \int_\Omega \rho^m\, dx - C  \int_\Omega \phi^{\frac{m-1}{m}}\, dx - C\nonumber\\
& \geq   \frac{C_1}{2} \int_\Omega \rho^m\, dx -  C \int_\Omega \phi^{2}\, dx - C\label{eq:coerc}
\end{align}
for some constant $C>0$ independent of $\rho$. Thus, $\sK(\cdot;\phi)$ is coercive on $L^m(\Omega)$. 

This implies that $\inf_{\rho \in \cP_\ac(\Omega) \cap L^m(\Omega)} \mathscr{K}(\rho;\phi)	>-\infty$ and that any
 minimizing sequence $\pseq{\rho_n}_n \subset L^m(\Omega) \cap \cP_\ac(\Omega)$ must be bounded in $L^m(\Omega)$.
Up to a subsequence, $\pseq{\rho_n}_n$ converges weakly in $L^m(\Omega)$ to $\rho\in L^m(\Omega) \cap \cP_\ac(\Omega)$  (since $\Omega$ is bounded, we have $\int_\Omega \rho(x)\, dx =1$).
The convexity and superlinearity of $f$  (Assumptions \ref{hyp:fMinimalAssumptions} and \ref{hyp:fGrowthCondition}) then imply the lower semicontinuity of $\sK$:
$$
\liminf_{n\to \infty} \int_\Omega f(\rho_n) - \rho_n \phi\, dx \geq \int_\Omega f(\rho) - \rho \phi\, dx 
$$
and so $\rho$ is a minimizer of $\mathscr{K}(\cdot;\phi)$.
Since $L^m(\Omega) \cap \cP_\ac(\Omega)$ is convex, the strict convexity of $f$, which gives the strict convexity of $ \mathscr{K}(\cdot;\phi)$, implies that this minimizer is unique.
It is easy to check that \eqref{eq:coerc} and \eqref{eq:infK} implies \eqref{eq:rhophibd}.
 \end{proof}

From now on, we denote by $\rho_\phi \in \cP_\ac(\Omega) \cap L^m(\Omega)$ the unique solution of \eqref{eq:minrhom}.
In order to show that the velocity potential $f'(\rho_\phi) - \phi$ is constant on the region where the density $\rho_\phi$ is positive, we first give the following classical result:
\begin{lemma}[optimality condition of \eqref{eq:minrhom}]\label{lemma:minrhom-optimalityCondition}
	Let $\rho_\phi \in \cP_\ac(\Omega) \cap L^m(\Omega)$ be the unique solution of \eqref{eq:minrhom}. Define $F(x) :=f'(\rho_\phi(x)) - \phi(x)$. Then 
	\begin{equation} \label{eq:rhoOptimalityCondition}
		0 \le \int_\Omega F(x) (\rho(x) - \rho_\phi(x))  \,dx	\qquad	\forall \, \rho\in\cP_\ac(\Omega) \cap L^m(\Omega).
	\end{equation}
\end{lemma}
\begin{proof}
For any $\rho \in \cP_\ac(\Omega) \cap L^m(\Omega)$ for which $\int_\Omega f(\rho)\,dx < +\infty$, and for $s \geq 0$ consider the convex combination $\rho_s :=(1-s)\rho_\phi + s \rho \in \cP_\ac(\Omega) \cap L^m(\Omega)$. We have
$ \sK(\rho;\phi) \leq \sK(\rho_s;\phi)$ for $s\in [0,1]$, that is 
	\begin{align*}
		0 \le \frac{ \sK(\rho_s;\phi) -  \sK(\rho;\phi)}{s} &= \int_\Omega \frac{f(\rho_s) - f(\rho_\phi)}{s} - (\rho - \rho_\phi)\phi \,dx \qquad \forall s\in (0,1).
	\end{align*}
The result follows by passing to the limit $s\to0^+$.	
\end{proof}

\begin{proof}[Proof of Proposition \ref{prop:rho} (ii)]
With the notation of Lemma \ref{lemma:minrhom-optimalityCondition}, 
if $F$ is not constant on $\set{\rho_\phi > 0}$, then there exists $\ell \in \bR$
such that the sets
	$$
	A :=\set{\rho_\phi > 0} \cap \set{F \ge -\ell},	\qquad B :=\set{\rho_\phi > 0} \cap \set{ F < -\ell },
	$$
	have both  positive measure.
	We move the mass that $\rho_\phi$ puts on $A$ onto the set $B$ where $F$ is smaller, in such a way that contradicts the optimality condition \eqref{eq:rhoOptimalityCondition}. Consider
	\[
	\ul \rho(x) :=
	\begin{cases}
		\rho_\phi(x) + \frac{1}{\abs{B}} \int_A \rho_\phi(y) \,dy 	&	\text{if } x \in B,	\\
		0														&	\text{if } x \in \cl\Omega \setminus B.
	\end{cases}
	\]
	The second term in the first condition is well-defined since $\abs{B} > 0$ and it is also strictly positive since $\rho_\phi > 0$ on $A$ and  $\abs{A}>0$. 
	The first term is also positive since $\rho_\phi > 0$ on $B$. It is clear $\ul \rho \in \cP_\ac(\Omega) \cap L^m(\Omega)$. The optimality condition provides a contradiction since
	\[
	0 \le \int_\Omega F(\ul\rho - \rho_\phi) \,dx = \int_B F \,dx  \Big(\frac{1}{\abs{B}} \int_A \rho_\phi \,dy \Big) - \int_A F \rho_\phi \,dx < \int_A (-\ell) \rho_\phi \,dy - \int_A F \rho_\phi \,dx \le 0.
	\]
We conclude that $F$ must be constant on the $\set{\rho_\phi > 0}$ and from now on we  denote by $-\ell$ this constant, that is
(see \eqref{eq:supportOfRho})
$$
			f'(\rho_\phi) - \phi = -\ell	\quad \rho_\phi\text{-a.e.}$$
\end{proof}

We now verify that this unique minimizer $\rho_\phi$ is actually a weak solution of the elliptic equation \eqref{eq:rhom}. This is a consequence of Proposition \ref{prop:rho}-(ii).
\begin{proof}[Proof of Proposition \ref{prop:rho} (iii)]
First, we claim that 
	\begin{equation}\label{eq:pressureEquality}
		f'(\rho_\phi) 
		= (\phi - \ell)_+.
	\end{equation}
Indeed, \eqref{eq:supportOfRho} implies that $f'(\rho_\phi) = \phi-\ell =  (\phi - \ell)_+$ in $\{\rho_\phi>0\}$ (the last equality follows from the fact that $f'(\rho_\phi)\geq 0$ due to \ref{hyp:fMinimalAssumptions}). Since $f'(0)=0$, we only need to check that $(\phi - \ell)_+=0$ a.e.\ in  $\{\rho_\phi=0\}$ in order to get \eqref{eq:pressureEquality}. 
Assume, by contradiction, that $A=\{ \phi>\ell\}\cap  \{\rho_\phi=0\}$ has positive Lebesgue measure.
We have $F(x) = f'(\rho_{\phi}(x)) - \phi(x) < -\ell$ in $A$ and so 
using the optimality condition \eqref{eq:rhoOptimalityCondition} with  $\rho :=\frac{1}{\abs{A}} \chi_A \in \cP_\ac(\Omega) \cap L^m(\Omega)$, we get:
$$
0\leq \int_\Omega F(x) (\rho(x)-\rho_\phi(x))\, dx < -\ell - \int_\Omega F(x) \rho_\phi(x)\, dx,
$$
which is a contraction since $ \int_\Omega F(x) \rho_\phi(x)\, dx = \int_\Omega (-\ell) \rho_\phi(x)\, dx  = -\ell$.
We thus have \eqref{eq:pressureEquality}.

Since $\phi\in H^1(\Omega)$, this equality \eqref{eq:pressureEquality} implies in particular that $f'(\rho_\phi)$ is in $H^1(\Omega)$ and that
$\na f'(\rho_\phi) = \chi_{\set{\phi > \ell}}\na \phi$. We deduce
$$ \rho_\phi \na f'(\rho_\phi) =  \rho_\phi \chi_{\set{\phi > \ell}} \na \phi   =  \rho_\phi   \na \phi $$
since $\rho_\phi>0$ implies $f'(\rho_\phi)>0$ which gives (by \eqref{eq:pressureEquality}) $\phi>\ell$.
We thus have
	\[
	\rho_\phi \nabla(f'(\rho_\phi) - \phi) = 0 \quad	\text{ a.e.\ in } \Omega,
	\]
which implies \eqref{eq:rhoWeakSolution}.
\end{proof}

Note that $f' \colon \bR_+ \to \bR_+$ is invertible (since $f$ is strictly convex \ref{hyp:fMinimalAssumptions}) and so 
 \eqref{eq:pressureEquality} gives 
	\begin{equation}\label{eq:densityFormula}
		\rho_\phi = (f')^{-1}((\phi - \ell)_+),
	\end{equation}
	where the Lagrange multiplier $\ell \in \bR$ is selected to enforce the mass constraint $\int_\Omega \rho_\phi \,dx = 1$.  Moreover, when $f$ satisfies  \ref{hyp:fStronglyConvex}, $(f')^{-1}$ is Lipschitz (with Lipschitz constant $\alpha$) so we can expect 
	the mapping $\phi \mapsto \rho_\phi$  to be Lipschitz  (a fact that we will need to prove the 
	uniqueness of the solution to \eqref{eq:phi}).
	However such a property does not follow immediately from \eqref{eq:densityFormula} since $\ell$ depends on $\phi$ as well. Nevertheless we can show:
\begin{proposition}[The map $\phi \mapsto \rho_\phi$]\label{prop:rhoStability}
 		
		Suppose that $f$ satisfies \ref{hyp:fStronglyConvex}, 
		(i.e. $f$ is $\alpha$-convex). 
		Then the mapping $\phi \mapsto \rho_\phi$ is Lipschitz $L^2(\Omega) \to L^2(\Omega)$. 
		More precisely, we have
		\begin{equation}\label{eq:liprho}
			\norm{\rho_{\phi_2} - \rho_{\phi_1}}_{L^2(\Omega)} \le \frac{2}{\alpha} \norm{\phi_2 - \phi_1}_{L^2(\Omega)}.
		\end{equation}
\end{proposition}
\begin{proof} 
	Since $f$ is $\alpha$-uniformly convex, we have
	\[
	f(u_2) - f(u_1) - f'(u_1)(u_2 - u_1) \ge \frac{\alpha}{2} \abs{u_2 - u_1}^2	\qquad	\forall \, u_2,u_1 \in [0,\infty).
	\]
	Combining the definition of $\sK$ (defined with $\phi = \phi_1$) with the above inequality gives
	\begin{align*}
		\sK(\rho_{\phi_2},\phi_1) - \sK(\rho_{\phi_1}, \phi_1) &:=\int_\Omega f(\rho_{\phi_2}) - f(\rho_{\phi_1}) \,dx - \int_\Omega \rho_{\phi_2} \phi_1 - \rho_{\phi_1} \phi_1 \,dx \\
		&\ge \int_\Omega (f'(\rho_{\phi_1}) - \phi_1)(\rho_{\phi_2} - \rho_{\phi_1}) \,dx + \frac{\alpha}{2}\int_\Omega \abs{\rho_{\phi_2} - \rho_{\phi_1}}^2 \,dx \\
		&\ge \frac{\alpha}{2}\norm{\rho_{\phi_2} - \rho_{\phi_1}}_{L^2(\Omega)}^2,
	\end{align*}
	where the last line follows from the optimality condition  \eqref{eq:rhoOptimalityCondition} applied to $\rho_{\phi_1}$.
		
	We rearrange this estimate and deduce
	\begin{gather*}
		\sK(\rho_{\phi_1}, \phi_1) + \frac{\alpha}{2}\norm{\rho_{\phi_2} - \rho_{\phi_1} }_{L^2(\Omega)}^2 \le \sK(\rho_{\phi_2}, \phi_1) = \sK(\rho_{\phi_2}, \phi_2) + \int_\Omega \rho_{\phi_2}(\phi_2 - \phi_1) \,dx \\
		\le \sK(\rho_{\phi_1}, \phi_2) + \int_\Omega \rho_{\phi_2}(\phi_2 - \phi_1) \,dx = \sK(\rho_{\phi_1}, \phi_1) + \int_\Omega \rho_{\phi_1}(\phi_1 - \phi_2) + \rho_{\phi_2}(\phi_2 - \phi_1) \,dx,
	\end{gather*}
	where the second line follows from $\rho_{\phi_2}$ minimizing $\sK$ when $\phi = \phi_2$. We simplify and use Cauchy-Schwarz to obtain
	\[
	\frac{\alpha}{2} \norm{\rho_{\phi_2} - \rho_{\phi_1}}_{L^2(\Omega)}^2 \le \int_\Omega (\rho_{\phi_2} - \rho_{\phi_1})(\phi_2 - \phi_1) \,dx \le \norm{\rho_{\phi_2} - \rho_{\phi_1}}_{L^2(\Omega)}\norm{\phi_2 - \phi_1}_{L^2(\Omega)},
	\]
	which implies the desired estimate.
\end{proof}

\section{Well-posedness of \texorpdfstring{\eqref{eq:PKS0}}{P\_\textepsilon}}\label{sec:exist}

The section is devoted to the proof of Theorem \ref{thm:existence}. The existence will be proven under only the assumptions \ref{hyp:fMinimalAssumptions} and \ref{hyp:fGrowthCondition} using a discrete time scheme based on the natural energy structure of the system -- the same energy structure that is essential to the limit $\eps\to 0^+$. The uniqueness will follow under the additional hypothesis \ref{hyp:fStronglyConvex} from the Lipschitz property of the map $\rho \mapsto \rho_\phi$ (Proposition \ref{prop:rhoStability}).  We note that this same Lipschitz property could also be used to prove the existence of a solution via a contraction mapping argument. However, the more constructive approach we follow has the advantage of 
giving us the energy dissipation inequality -- which plays a crucial role in the limit $\eps\to0^+$ -- in a straightforward manner.

Throughout this section, $f$ satisfies \ref{hyp:fMinimalAssumptions} and \ref{hyp:fGrowthCondition} with $m>2$.

\subsection{A minimizing movements scheme}
We fix a time horizon $T > 0$, a small time step $\tau \in (0,T)$ (destined to go to $0$) and we denote $N_\tau :=\ceil{T/\tau} \in \bN$.
We set $(\rho^0_\tau, \phi^0_\tau) :=(\rho_{\phi_\init}, \phi_\init)$ and then solve iteratively the following minimization problems:
\begin{equation}\label{eq:minimizingMovementsScheme}
	(\rho^n_\tau,\phi^n_\tau) \in \argmin_{(\rho,\phi) \in L^m(\Omega) \times H^1(\Omega)}\left\{ \eps\frac{\norm{ \phi-\phi^{n-1}_\tau}_{L^2(\Omega)}^2}{2\tau} + \sG(\rho,\phi) \right\}	\qquad	\text{for } n = 1,2,\dots,N_\tau,
\end{equation}
where
\begin{equation}\label{eq:minMoveEnergy}
	\sG(\rho,\phi) :=
	\begin{cases}
		\displaystyle \frac{1}{\eps}\int_\Omega f(\rho)-  \rho \phi + \frac {\sigma} {2}   |\phi|^2 \,dx + \frac{\eps} {2} \int_\Omega |\na \phi|^2\, dx & \text{if } (\rho,\phi) \in (\cP_\ac(\Omega) \cap L^m(\Omega)) \times H^1(\Omega), \\
		+\infty	&\text{otherwise in } L^m(\Omega) \times H^1(\Omega).
	\end{cases}
\end{equation}

We first need to check that this iterative scheme is well-defined, which requires the lower semicontinuity and coercivity of the functional. More precisely, we can show:
\begin{lemma}\label{lem:cont}

Given $\xi \in L^2(\Omega)$, the functional 
\begin{equation} \label{eq:minMoveOverallEnergy}
	L^m(\Omega) \times H^1(\Omega) \ni (\rho,\phi) \mapsto \eps\frac{\| \phi-\xi\|_{L^2(\Omega)}^2}{2\tau} + \sG(\rho,\phi) \in \bR \cup \set{+\infty}
\end{equation}
is  sequentially lower semicontinuous for the weak convergence in  $L^m(\Omega)\times H^1(\Omega)$ and we have
\begin{align} 
\sG(\rho,\phi) 
& \geq 
   \frac{c}{\eps} \left[  \int_\Omega \rho^m\, dx  +  \int_\Omega   |\phi|^2 \,dx
+\eps^2 \int_\Omega |\na \phi|^2\, dx \right] - \frac C \eps  \label{eq:Gcoerc}
\end{align}
for some constants $c>0$ and $C>0$ that depend only on $\Omega$, $f$, and $\sigma$.
\end{lemma}

\begin{proof}[Proof of Lemma \ref{lem:cont}]
Given any sequence  $\seq{(\rho_n, \phi_n)}_n$ such that  $\rho_n\rightharpoonup \rho$ weakly in $L^m(\Omega)$ and $\phi_n\rightharpoonup \phi$ weakly in $H^1(\Omega)$ (and thus  strongly in $L^2(\Omega)$), it is easy to check that these convergences and the lower semicontinuity of $\sK$ give the lower semicontinuity of $\sG$:
$$ \liminf_{n\to\infty}  \sG(\rho_n,\phi_n)\geq  \sG(\rho,\phi), $$
which is the first part of the lemma.

Next, we write, using \eqref{eq:coerc}:
\begin{align} 
\sG(\rho,\phi) 
& \geq 
   \frac{C_1}{2\eps}  \int_\Omega \rho^m\, dx -  \frac{C}{\eps} \int_\Omega \phi^{\frac{m}{m-1}}\, dx - \frac{C}{\eps}
 +  \frac{1}{\eps}\int_\Omega \frac {\sigma} {2}   |\phi|^2 \,dx + \frac{\eps} {2} \int_\Omega |\na \phi|^2\, dx\nonumber \\
 & \geq 
   \frac{C_1}{2\eps}  \int_\Omega \rho^m\, dx +\frac {\sigma} {4\eps}\int_\Omega    |\phi|^2 \,dx    + \frac{\eps} {2} \int_\Omega |\na \phi|^2\, dx + 
\frac 1 \eps\int_\Omega 
   \frac {\sigma} {4}  |\phi|^2- C 
   \phi^{\frac{m}{m-1}}\, dx - \frac{C}{\eps}.
\end{align}
Since $\frac{m}{m-1}<2$, the function $u \mapsto   \frac {\sigma} {4} u^2-C u^{\frac{m}{m-1}}$ is bounded below and so \eqref{eq:Gcoerc} follows.
\end{proof}

\begin{proposition}[existence of minimizers]\label{prop:minMove-existenceUniqueness}
We have:
	\begin{enumerate}[label=(\roman*)]
		\item For any $\xi \in L^2(\Omega)$ and  $\tau \in (0,T)$, there exists a minimizer $(\ul\rho,\ul\phi) \in (\cP_\ac(\Omega) \cap L^m(\Omega)) \times H^1(\Omega)$ of the energy \eqref{eq:minMoveOverallEnergy}. 
		\item Any minimizer is of the form $(\rho_{\ul\phi}, \ul\phi)$, where $\rho_{\ul\phi}$ is the minimizer of \eqref{eq:minrhom} with $\phi = \ul\phi$.
	\end{enumerate}
\end{proposition}
Note that when $f$ satisfies \ref{hyp:fStronglyConvex}, 
one can show that $\sG$ is strictly convex (at least for small enough $\tau$) and that the minimizer is unique. Such a fact is however not  necessary for our analysis. 
In the sequel, we arbitrarily pick one solution of \eqref{eq:minimizingMovementsScheme} for each $n$.

\begin{proof}[Proof of Proposition \ref{prop:minMove-existenceUniqueness}]
Lemma \ref{lem:cont} implies that the functional \eqref{eq:minMoveOverallEnergy} is bounded below 
and that any minimizing sequence is bounded in $L^m(\Omega) \times H^1(\Omega)$ and hence converges weakly (up to a subsequence) to a minimizer. Since the limit is still in $(\cP_\ac(\Omega) \cap L^m(\Omega)) \times H^1(\Omega)$, the existence of a minimizer follows.

Furthermore, given any minimizer $(\ul\rho, \ul\phi)$ of  \eqref{eq:minMoveOverallEnergy}, we have  (using $(\rho_{\ul\phi}, \ul\phi)$ as a competitor):
$$\eps\frac{\| \ul\phi-\xi\|_{L^2(\Omega)}^2}{2\tau} + \sG(\ul\rho, \ul\phi) \le  \eps\frac{\| \ul\phi-\xi\|_{L^2(\Omega)}^2}{2\tau} +\sG(\rho_{\ul\phi}, \ul\phi)$$
which, after simplifying the terms that are identical on both sides, reduces to
$$
\frac{1}{\eps}\int_\Omega f(\ul\rho)-  {\ul\rho}\, {\ul\phi} \leq \frac{1}{\eps}\int_\Omega f(\rho_{\ul\phi})-  \rho_{\ul\phi}\ul\phi .
$$
Proposition \ref{prop:rho} (i) implies $\ul\rho = \rho_{\ul\phi}$.
 \end{proof}

\medskip

In view of the previous proposition, with $\tau \in (0,T)$ fixed,  we can use \eqref{eq:minimizingMovementsScheme} to define a sequence   $\seq{(\rho^n_\tau, \phi^n_\tau)}_n \subset (\cP_\ac(\Omega)\cap L^m(\Omega)) \times H^1(\Omega)$. 
A straightforward perturbation argument of $\sG$ along the $\phi$ variable with the $\rho$ variable fixed gives the following  Euler-Lagrange equation:
\begin{lemma}[optimality conditions of \eqref{eq:minimizingMovementsScheme}] \label{lemma:minMove-optimalityConditions}
	For all $\tau \in (0,T)$ and any step $n \in \set{1,\dots,N_\tau}$ of \eqref{eq:minimizingMovementsScheme},
	we have
	\begin{gather}
		0 = \int_\Omega \frac{\varepsilon}{\tau}(\phi^n_\tau - \phi^{n-1}_\tau)\psi + \varepsilon \nabla\phi^n_\tau \cdot \nabla \psi + \frac{1}{\varepsilon}(\sigma\phi^n_\tau - \rho^n_\tau)\psi \,dx	\qquad	\forall \, \psi\in H^1(\Omega).\label{eq:minMove-phiOptimality}
	\end{gather}
\end{lemma}

\medskip

Next, we derive the classical estimates for discrete gradient flows. 
\begin{lemma} \label{lemma:gradientFlowEstimates}
	For any step size $\tau \in (0,T)$ and $n \in \set{1,\dots,N_\tau}$  we have:
	\begin{enumerate}[label=(\roman*)]
		\item (nonincreasing energy): $\sG(\rho^n_\tau, \phi^n_\tau) \le \sG(\rho^{n-1}_\tau, \phi^{n-1}_\tau) \le \sG(\rho_{\phi_\init}, \phi_\init)$;
		
		\item (discrete dissipation estimate): $\displaystyle \sum_{j=1}^{n} \frac{\varepsilon}{2\tau} \norm{\phi^j_\tau - \phi^{j-1}_\tau}_{L^2(\Omega)}^2 \le \sG(\rho_{\phi_\init}, \phi_\init) - \sG(\rho_\tau^n, \phi_\tau^n).$ 
		
	\end{enumerate}
\end{lemma}
\begin{proof}
	We use $(\rho^{n-1}_\tau, \phi^{n-1}_\tau)$ as a competitor in  the minimization problem \eqref{eq:minimizingMovementsScheme} at step $n$  to get:
		\begin{equation}\label{eq:energyn}
	\varepsilon\frac{\norm{\phi^n_\tau - \phi^{n-1}_\tau}_{L^2(\Omega)}^2}{2\tau} + \sG(\rho^n_\tau, \phi^n_\tau) \le \sG(\rho^{n-1}_\tau, \phi^{n-1}_\tau).
	\end{equation}
This implies (i) and also gives (after summation)
\begin{equation}\label{eq:disdis}
\sG(\rho^n_\tau, \phi^n_\tau) + 	\sum_{j=1}^n \frac{\varepsilon \norm{\phi^j_\tau - \phi^{j-1}_\tau}_{L^2(\Omega)}^2}{2\tau} \le \sG(\rho_{\phi_\init}, \phi_\init) 
\end{equation}
which is (ii).
	
 \end{proof}

\label{subsection:interpolants&uniformEstimates}

We now define the usual piecewise-constant and piecewise-affine interpolations: 
For convenience, we denote $t^n_\tau :=n\tau$
and define
\begin{enumerate}
	\item The \emph{piecewise-constant interpolant}: 
	\begin{equation} \label{eq:pwConstantInterp}
		\ol\phi_\tau(t,x) :=
		\begin{cases}
			\phi_\tau^n(x)	&	\text{if } t \in (t^{n-1}_\tau, t^n_\tau] \text{ for some } n \in \set{1,\dots,N_\tau},	\\
			\phi_\init(x)	&	\text{if } t = 0.
		\end{cases}
	\end{equation}
	The interpolant $\ol\rho_\tau \colon [0,N_\tau \tau] \to \cP_\ac(\Omega) \cap L^m(\Omega)$ is defined analogously.
	
	\item The \emph{piecewise-affine interpolant}: 
	\begin{equation}\label{eq:pwAffineInterp}
		\wh{\phi}_\tau(t,x) :=
		\begin{cases}
			(1-\frac{t-t_\tau^{n-1}}{\tau})\phi_\tau^{n-1} (x)+ \frac{t-t_\tau^{n-1}}{\tau}\phi_\tau^n	(x) &	\text{if } t \in (t_\tau^{n-1}, t_\tau^n] \text{ for some } n \in \set{1,\dots,N_\tau},	\\
			\phi_\init(x)	&	\text{if } t=0.
		\end{cases}
	\end{equation}
\end{enumerate}

With these notations, we have:
\begin{proposition}[semidiscretization in time of \eqref{eq:PKS0}] \label{prop:semidiscretizationInTime}
Suppose $\phi_\init \in H^1(\Omega)$. Then for any $\psi \in L^2(0,T;H^1(\Omega))$
		\begin{equation}\label{eq:semidiscretization-parabolicEq}
			\begin{gathered}
				\int_0^T\int_\Omega \varepsilon \pd {\wh{\phi}_\tau} t  \psi + \varepsilon \nabla \ol{\phi}_\tau \cdot \nabla \psi + \frac{\sigma}{\varepsilon} \ol{\phi}_\tau \psi - \frac{1}{\varepsilon}\ol{\rho}_\tau \psi \,dx \,dt 
				=  
				0.
			\end{gathered}
		\end{equation}
\end{proposition}
\begin{proof}
 \medskip
 
We use \eqref{eq:minMove-phiOptimality} to write for all $t \in (t_\tau^{n-1}, t_\tau^n)$: 
 	\begin{align*}
	\int_\Omega \varepsilon\pd{\wh{\phi}_\tau} t (t) \psi(t) \,dx = \int_\Omega \varepsilon\frac{\phi_\tau^n - \phi_\tau^{n-1}}{\tau} \psi(t) \,dx  
& = \int_\Omega -\varepsilon \nabla \phi_\tau^n \cdot \nabla \psi(t) - \frac{\sigma}{\varepsilon}\phi_\tau^n\psi(t) + \frac{1}{\varepsilon}\rho^n_\tau\psi(t) \,dx\\
& = \int_\Omega -\varepsilon \nabla \ol \phi_\tau(t)  \cdot \nabla \psi(t) - \frac{\sigma}{\varepsilon}\ol \phi_\tau(t) \psi(t) + \frac{1}{\varepsilon}\ol \rho_\tau(t)\psi(t) \,dx.
	\end{align*}
The result follows by integrating in time and summing over $n$.

\end{proof}

\subsection{Uniform-in-\texorpdfstring{$\tau$}{\texttau} estimates}

All the estimates derived in this section are independent of $\tau$ but depend on $\eps$ and other constants of the problem even though this dependence is not explicitly written.

The estimates of Lemma \ref{lemma:gradientFlowEstimates} imply:
\begin{lemma}[time regularity of chemoattractant interpolants] \label{lemma:timeRegularity-phi}
	There exist a constant $C$ independent of $\tau$ 
	 such that for any $\tau \in (0,T)$ and $0 \le s \le t \le T$,
	\begin{enumerate}[label=(\roman*)]
		
		\item $\norm{\pd{\wh{\phi}_\tau} t}_{L^2(0,T;L^2(\Omega))} < C$ and  $\norm{\wh{\phi}_\tau(t) - \wh{\phi}_\tau(s)}_{L^2(\Omega)} \le C\sqrt{t-s}$;
		
		\item $\norm{\ol{\phi}_\tau(t) - \wh{\phi}_\tau(t)}_{L^2(\Omega)} \le C\sqrt{\tau}$.
	\end{enumerate}
\end{lemma}
Note that the piecewise-constant interpolant $\ol\phi_\tau$ is, of course,  not continuous in time, but it satisfies an almost-H\"older estimate
$\norm{\ol{\phi}_\tau(t) - \ol{\phi}_\tau(s)}_{L^2(\Omega)} \le C  \sqrt{t-s + \tau}$ with the same $C$ as in the lemma.

\begin{proof}
First we note that
$$ \sG(\rho_{\phi_\init}, \phi_\init) - \inf_{(\rho,\phi) \in \cP_\ac(\Omega) \times H^1(\Omega)} \sG(\rho,\phi)  <\infty
$$
and so the discrete energy dissipation inequality (Lemma \ref{lemma:gradientFlowEstimates} (ii)) gives
	\[
	\int_0^T \norm{\pd{\wh{\phi}_\tau (t)} t}_{L^2(\Omega)}^2 \,dt = \sum_{j=1}^{N_\tau} \int_{t^{j-1}_\tau}^{t_\tau^j \wedge T} \norm{\frac{\phi_\tau^j - \phi_\tau^{j-1}}{\tau}}_{L^2(\Omega)}^2 \,dt \leq \sum_{j=1}^{N_\tau} \frac{1}{\tau} \norm{\phi_\tau^j - \phi_\tau^{j-1}}_{L^2(\Omega)}^2  \le C.
	\]
which implies (i).	

	To prove (ii), we note that the piecewise-constant and piecewise-affine interpolants agree at times on the mesh, and that the mesh has size $\tau$: If $t \in (t_\tau^{n-1}, t_\tau^n]$ we have $\ol{\phi}_\tau(t) = \phi_\tau^n = \wh{\phi}_\tau(t_\tau^n)$ and so
	\[
	\norm{\ol{\phi}_\tau(t) - \wh{\phi}_\tau(t) }_{L^2(\Omega)} = \norm{\wh{\phi}_\tau(t_\tau^n) - \wh{\phi}_\tau(t)}_{L^2(\Omega)} \le C \sqrt{t_\tau^n - t} \le C\sqrt{\tau}.
	\]
\end{proof}

Next, the boundedness of the energy (Lemma \ref{lemma:gradientFlowEstimates} (i)) together with the coercivity inequality \eqref{eq:Gcoerc} give the following estimates:
\begin{lemma} \label{lemma:aPioriEstimates-from-Energy}
	There exists a constant $C>0$, independent of $\tau \in (0,T)$,  such that
	\begin{enumerate}[label = (\roman*)]
		\item $\sup_{t \in [0,T]}\norm{\wh{\phi}_\tau(t)}_{H^1(\Omega)} \le C$;
		\item $\sup_{t \in [0,T]}\norm{\ol{\rho}_\tau (t)}_{L^m(\Omega)} \le C$;
	\end{enumerate}	
\end{lemma}
\begin{proof}
Inequality \eqref{eq:Gcoerc} and Lemma \ref{lemma:gradientFlowEstimates}-(i) and imply
	\begin{equation}\label{eq:pwConstantEst-GLB}
		 \norm{\ol{\rho}_\tau(t)}_{L^m(\Omega)}^m	 +  \norm{ \ol{\phi}_\tau(t)}_{H^1(\Omega)}^2  \le C \sG(\ol{\rho}_\tau(t),\ol{\phi}_\tau(t)) + C \le C \sG(\rho_{\phi_\init}, \phi_\init) + C,  
	\end{equation}
	and the result (ii) follows. For (i), we note that $\wh{\phi}_\tau(t)$ is a convex combination of $\ol{\phi}_\tau(t-\tau)$ and $\ol{\phi}_\tau(t)$ and use the inequality above.
\end{proof}

\subsection{Passage to the limit \texorpdfstring{$\tau \to 0^+$}{\texttau\textrightarrow 0⁺} in the semidiscretization}
We now use these estimates to pass to the limit in \eqref{eq:semidiscretization-parabolicEq}   along some subsequence.
Note that we will not pass to the limit in the relation $\ol{\rho}_\tau = \rho_{\ol\phi_\tau}$ but instead show directly that the weak limit of $\pseq{\ol{\rho}_\tau}_\tau$ minimizes $\sK(\cdot,\phi(t))$ at each fixed $t$, so it is therefore equal to $\rho_{\phi(t)}$ by using the uniqueness of minimizers of the minimization problem \eqref{eq:minrhom}.
\begin{proposition}[convergence of the interpolants]\label{prop:convergenceOfInterpolants}
	Let $\seq{\tau_k}_k$ be a positive sequence converging to $0$. There exists a subsequence of $\seq{\tau_k}_k$ (not relabeled) and nonnegative functions 
	\begin{gather*}
		\phi \in H^1(0,T; L^2(\Omega)) \cap L^\infty(0,T; H^1(\Omega)) \quad\text{and}\quad
		\rho \in L^\infty(0,T; L^m(\Omega))
	\end{gather*}
	such that:
	\begin{enumerate}[label = (\roman*)]
		\item $\pd{\wh{\phi}_{\tau_k}} t \rightharpoonup \pd\phi t$ weakly in $L^2(0,T;L^2(\Omega))$;
		\item $\nabla \ol{\phi}_{\tau_k} \rightharpoonup \nabla\phi$ weakly-$\ast$ in $L^\infty(0,T; L^2(\Omega)^d)$;
		\item $\ol{\phi}_{\tau_k} \to \phi$ in $L^\infty(0,T; L^2(\Omega))$ and $\wh{\phi}_{\tau_k} \to \phi$ in $C([0,T]; L^2(\Omega))$;
		\item $\ol{\rho}_{\tau_k} \rightharpoonup \rho$ 
		weakly-$\ast$ in $L^\infty(0,T; L^m(\Omega))$;
	\end{enumerate} 
\end{proposition}

\begin{proof}

Lemma \ref{lemma:aPioriEstimates-from-Energy} immediately gives the convergences (ii) and (iv), while (i) is a consequence of 
Lemma \ref{lemma:timeRegularity-phi}-(i).

	Assertion (iii) follows from Arzel\`{a}-Ascoli. Indeed, Lemma \ref{lemma:timeRegularity-phi}-(i) shows the family $\pseq{\wh{\phi}_{\tau_k}}_k$ is uniformly bounded in $C^{0,1/2}([0,T];L^2(\Omega))$
and Lemma \ref{lemma:aPioriEstimates-from-Energy}-(i) shows that $\pseq{\wh{\phi}_{\tau_k}(t)}_k$ is bounded in $H^1(\Omega)$ (and thus precompact in $L^2(\Omega)$) for each $t \in [0,T]$.
Arzel\`{a}-Ascoli provides strong convergence of a subsequence of $\pseq{\wh{\phi}_{\tau_k}}_k$ in $C([0,T]; L^2(\Omega))$. Along this subsequence, we also have $\ol{\phi}_{\tau_k} \to \phi$ in $L^\infty(0,T; L^2(\Omega))$ as a consequence of Lemma \ref{lemma:timeRegularity-phi}-(ii). 
	
\end{proof}

When $f$ satisfies \ref{hyp:fStronglyConvex}, the fact that the map $\phi \mapsto \rho_\phi$ is Lipschitz $L^2(\Omega) \to L^2(\Omega)$ can be used to show that $\rho = \rho_\phi$. 
This assumption on $f$ is however not necessary as we can show directly that any weak limit $\rho$
must be solution of the minimization problem \eqref{eq:minrhom}:

\begin{proposition}\label{prop:minimizerAtEachTime}
	For $\rho$ obtained as limit of $\pseq{\ol{\rho}_{\tau_k}}_k$   as in Proposition \ref{prop:convergenceOfInterpolants}-(iv), we have:
	\begin{enumerate}[label=(\roman*)]
		\item for a.e.\ $t \in [0,T]$, we have $ \rho(t) =  \rho_{\phi(t)}$ a.e.\ in $\Omega$;
		\item $\rho=\rho_{\phi} \in C([0,T]; L^m(\Omega)\wk)$.
	\end{enumerate}
\end{proposition}
As a consequence of (ii), we choose always to work with this continuous representative. 

\begin{proof}
Proposition \ref{prop:minMove-existenceUniqueness} implies that 
for any given $t\in[0,T]$, $\ol \rho_\tau(t)$ is the unique minimizer of $\nu\mapsto \int_\Omega f(\nu)-\nu \ol \phi_\tau(t)\, dx$; that is, 
$$
\int_\Omega f(\ol \rho_\tau(t))-\ol \rho_\tau(t) \ol \phi_\tau(t)\, dx \leq \int_\Omega f(\nu)-\nu \ol \phi_\tau(t)\, dx \quad \forall \nu \in  \cP_\ac(\Omega) \cap L^m(\Omega).
$$
Since $\pseq{\ol \rho_{\tau_k}(t)}_k$ is bounded in $L^m(\Omega)$ by Lemma \ref{lemma:aPioriEstimates-from-Energy}-(ii), it converges weakly, up to a not relabeled subsequence, to some $\mu \in \cP_\ac(\Omega) \cap L^m(\Omega)$ and since $\ol{\phi}_{\tau_k}(t) \to \phi(t)$  strongly  in $L^2(\Omega)$, we can pass to the limit in the inequality above to get
$$
\int_\Omega f(\mu)-\mu  \phi (t)\, dx \leq \liminf_{k\to\infty} \int_\Omega f(\ol \rho_{\tau_k}(t))-\ol \rho_{\tau_k}(t) \ol \phi_{\tau_k}(t)\, dx\leq  \int_\Omega f(\nu)-\nu\phi(t)\, dx \quad \forall \nu \in  \cP_\ac(\Omega) \cap L^m(\Omega).
$$
The uniqueness from Proposition \ref{prop:rho}-(i) thus gives $\mu = \rho_{\phi(t)}$, and by starting this argument with an arbitrary subsequence of $\pseq{\ol{\rho}_{\tau_k}}_k$, it follows that the whole sequence $\pseq{\ol \rho_{\tau_k}(t)}_k$ converges to $\rho_{\phi(t)}$ weakly in $L^m(\Omega)$. Since the $L^m(\Omega)$-bound in Lemma \ref{lemma:aPioriEstimates-from-Energy}-(ii) is uniform-in-$t$, this pointwise convergence in time to $\rho_\phi$ implies the weak-$\ast$ convergence in $L^\infty(0,T;L^m(\Omega))$ to $\rho_\phi$. Since we started with a sequence that converged weakly-$\ast$ to $\rho$, we deduce $\rho = \rho_\phi$.

Finally, since $[0,T] \ni t \mapsto \phi(t) \in L^2(\Omega)$ is strongly continuous, we have for any $t \in [0,T]$ that $\lim_{s \to t} \norm{\phi(s) - \phi(t)}_{L^2(\Omega)} = 0$. The same argument as above can be used to show that  $\rho_{\phi(s)} \rightharpoonup \rho_{\phi(t)}$ weakly in $L^m(\Omega)$ as $s \to t$, which gives (ii).
\end{proof}

We can now prove the first part of Theorem \ref{thm:existence}
\begin{proof}[Proof of Theorem \ref{thm:existence} -- Existence]
The convergences of Proposition \ref{prop:convergenceOfInterpolants} allow us to pass to the limit in \eqref{eq:semidiscretization-parabolicEq} to get
$$\int_0^T\int_\Omega \varepsilon \pd \phi t  \psi + \varepsilon \nabla \phi  \cdot \nabla \psi + \frac{\sigma}{\varepsilon}  \phi  \psi - \frac{1}{\varepsilon} \rho \psi \,dx \,dt 				=  				0
$$
for any $\psi \in L^2(0,T; H^1(\Omega))$, which implies (in some weak sense)
 $$\varepsilon \pd \phi t   - \varepsilon \Delta \phi   = \frac{1}{\varepsilon} (\rho- \sigma\phi  ).
 $$
Since $ \pd\phi t\in L^2(0,T;L^2(\Omega))$, this equality actually implies that  $\lap\phi \in L^2(0,T;L^2(\Omega))$,  and so $ \phi$ is a strong solution. Since $\pseq{\wh\phi_{\tau_k}}_k$ converges strongly in $C([0,T]; L^2(\Omega))$, the limit satisfies the initial condition $\phi(0)=\phi_{in}$.
Proposition \ref{prop:minimizerAtEachTime} gives for a.e.\ $t \in [0,T]$ that $\rho(t) = \rho_{\phi(t)}$, and so $(\rho,\phi)$ is a solution of \eqref{eq:phi}.
\medskip

It remains to show that $\phi$ satisfies the energy inequality \eqref{eq:energy0}.	The discrete energy inequality \eqref{eq:disdis} can be written as
$$
\sG(\ol \rho_\tau(t), \ol \phi_\tau(t) ) + \frac 1 2	\int_0^t  \varepsilon \norm{\pa_t \wh \phi_\tau(s)}_{L^2(\Omega)}^2 \, ds \le \sG(\rho_{\phi_\init}, \phi_\init) 
$$
and the weak convergence of $\pseq{\pa_t \wh \phi_{\tau_k}}_k$ and  the lower semicontinuity of $\sG(\cdot,\cdot)$ w.r.t.\ joint weak $L^m(\Omega)$ convergence and weak $H^1(\Omega)$ convergence then gives 
	\[
	\sG(\rho(t), \phi(t)) +  \frac 1 2	\int_0^t  \varepsilon \norm{\pa_t  \phi(s)}_{L^2(\Omega)}^2 \, ds  \le \sG(\rho_{\phi_\init}, \phi_\init).	
	\]
	Since $\rho(t) = \rho_{\phi(t)}$ for all $t$ 
	(see Proposition \ref{prop:minimizerAtEachTime}), 
	we deduce
		\[
	\sE_\varepsilon(\phi(t))  + \frac{1}{2}\int_0^t \varepsilon \norm{\pd \phi t (s)}_{L^2(\Omega)}^2 \,ds \le \sE_\varepsilon(\phi_\init), 
	\]
	where we recall that $\sE_\eps (\phi)  = \sG(\rho_\phi,\phi).$
\end{proof}

Note that the energy inequality derived here is not optimal due to the  factor of $\frac{1}{2}$, but this is sufficient for our forthcoming aim of deriving volume-preserving mean-curvature flow. This inequality can be improved to its optimal version (without the factor of $\frac{1}{2}$) without much additional effort by using the De Giorgi variational interpolant of $\pseq{\phi_\tau^n}_{n=1}^{N_\tau}$ (see \cite[Section 6]{Mim17}, \cite[Section 4.3]{2015_CarrilloKinderlehrer-etal_2dKS-gradientFlow}
or \cite[Chapter\ 3]{AGS08}).

\subsection{Proof of Theorem \ref{thm:existence}: Uniqueness of solutions}
We now turn to the uniqueness of the solution, which we prove here under the stronger assumption 
 that $f$ satisfies \ref{hyp:fStronglyConvex} ($f$ is $\alpha$-uniformly convex for some $\alpha>0$). 
We recall that under this assumption, Proposition \ref{prop:rhoStability}-(ii) provides that the map $\phi \mapsto \rho_\phi$ is Lipschitz continuous $L^2(\Omega) \to L^2(\Omega)$. This Lipschitz continuity now makes the argument  classical.

Given two solutions $ \phi^1$, $\phi^2$ of  \eqref{eq:phi} with initial data $\phi^1_\init$ and $\phi^2_\init$  we have:
\begin{align*}
	& \frac{1}{2}\fd{} t \norm{\phi^1(t) - \phi^2(t)}_{L^2(\Omega)}^2 +  \norm{\nabla(\phi^1(t) - \phi^2(t))}_{L^2(\Omega)}^2 \\
	&\qquad\qquad  = -\frac{\sigma}{\eps^2}\norm{\phi^1(t) - \phi^2(t)}_{L^2(\Omega)}^2 + \frac{1}{\eps^2}\int_\Omega(\rho_{\phi^1(t)} - \rho_{\phi^2(t)})(\phi^1(t) - \phi^2(t))\,dx \\
	&\qquad \qquad \leq -\frac{\sigma}{\eps^2}\norm{\phi^1(t) - \phi^2(t)}_{L^2(\Omega)}^2  + \frac{1}{\eps^2}\norm{\phi^1(t) - \phi^2(t)}_{L^2(\Omega)}\norm{\rho_{\phi^1}(t) -\rho_{ \phi^2}(t)}_{L^2(\Omega)} \\
	& \qquad\qquad  \leq -\frac{\sigma}{\eps^2}\norm{\phi^1(t) - \phi^2(t)}_{L^2(\Omega)}^2  + \frac{2}{\alpha\eps^2}\norm{\phi^1(t) - \phi^2(t)}_{L^2(\Omega)}^2 
\end{align*}
where the last line follows from the Lipschitz continuity  \eqref{eq:liprho}.
Gronwwall's lemma gives
\[
\norm{\phi^1(t) - \phi^2(t)}_{L^2(\Omega)}^2 \le \norm{\phi^1_\init - \phi^2_\init}^2_{L^2(\Omega)} \exp\left(\frac{4 - 2\alpha\sigma}{\alpha\eps^2}t \right)	\qquad	\forall \, t \in [0,T].
\]
When the initial data are the same $\phi_\init^1 = \phi_\init^2$, we get $\phi^1(t) = \phi^2(t)$ a.e.\ in $\Omega$ for all $t \in [0,T]$. 

\subsection{Proof of Theorem \ref{thm:existence}: $L^\infty$-bound on solutions}
We now complete the proof of Theorem \ref{thm:existence} by showing that when the initial condition is in $L^\infty$, the solution remains in $L^\infty$ for all $t \in [0,T]$.
For this proof, we make the additional assumption (which is compatible with \ref{hyp:fGrowthCondition}) that $u \mapsto f'(u)$ has  at least linear growth for large $u$ and so 
$$(f')^{-1}(v) \leq C_2 v +R_2.$$
Recalling that $\rho_{\phi(t)} = (f')^{-1}((\phi(t) - \ell(t))_+)$, we deduce
$$
0\leq \rho_{\phi(t)} \leq C_2 (\phi(t) - \ell(t))_+ + R_2 \leq C_2 \phi(t) + (-\ell)_+ + R_2
$$
(since $\phi(t)\geq 0$) and the condition $\int_\Omega \rho_{\phi(t)}\, dx =1$ implies $(-\ell)_+ \leq f'\left(\frac{1}{|\Omega|}\right)$.
In particular, $\phi$ solves
$$
\pa_t \phi - \Delta \phi \leq \frac 1 {\eps^2} \left( (C_2-\sigma) \phi(t) + f'\left(\frac{1}{|\Omega|}\right)  + R_2\right)
$$
and a simple barrier argument implies (when $C_2\neq \sigma$)
$$
\|\phi(t)\|_{L^\infty(\Omega)} \leq \|\phi(0)\|_{L^\infty(\Omega)} e^{\frac{C_2-\sigma}{\eps^2} t}
+ \frac{f'\left(\frac{1}{|\Omega|}\right)  + R_2}{C_2-\sigma}\left(e^{\frac{C_2-\sigma}{\eps^2} t}-1\right), 
$$
which gives the result.

\section{\texorpdfstring{$\Gamma$}{\textGamma}-convergence of the energy $\sJ_\varepsilon$ (Proof of Theorem \ref{thm:Gamma})}
\label{sec:Gamma}
From here through the remainder of the paper, we assume that $f$ satisfies \ref{hyp:fSecondBounds} and 
we work with the energy $\sJ_\varepsilon$ defined by \eqref{eq:J0}, which can also be written as
\begin{equation}\label{eq:Jepsbis}
\sJ_\eps(\phi) = \frac 1 \eps  \left\{  \int_\Omega W(\rho_\phi) + \frac 1 {2\sigma}  (\rho_\phi-\sigma \phi )^2\, dx \right\}+\eps  \int_\Omega \frac 1 2 |\na \phi |^2\, dx \qquad \forall \phi \in H^1(\Omega).
\end{equation}
We recall, $\sJ_\varepsilon$ differs from $\sE_\varepsilon$ only by a constant. 
With the double-well potential $W_\sigma$ defined by \eqref{eq:g}, we have \eqref{eq:JW}; that is, 
\begin{equation}\label{eq:lowerbound}
	\sJ_\eps(\phi) \geq  \sF_\eps(\phi ) \qquad \forall \phi \in H^1(\Omega),  
\end{equation}
where
\begin{equation}\label{eq:ModMortF}
	\sF_\eps(\phi ) 
	:=  \frac 1 \eps \int_\Omega W_\sigma(\sigma \phi)  \, dx  +\eps \int_\Omega \frac1  2 |\na \phi|^2\, dx \geq 0.
	%
\end{equation}

This might seem like a very crude bound since replacing 
$ \inf_{\rho\in \cP_\ac(\Omega) \cap L^m(\Omega)} \int  W(\rho) + \frac{1}{2\sigma}(\rho-\sigma\phi)^2\, dx$ by $\int W_\sigma(\sigma \phi)\, dx $ 
completely ignores the constraint $\int_{\Omega} \rho_\phi \, dx =1 $.
However, 
the main theorem of this section, namely the $\Gamma$-convergence of $\sJ_\eps$ to $\sJ_0$, shows that
these two functionals have the same $\Gamma$-limit.
A key aspect of this proof is the fact that the functional $\sF_\eps$
is a classical Modica-Mortola functional. Indeed, we have
(see \cite[Lemma 3.1]{M23}):
\begin{lemma}[double-well potential]\label{lem:gF1}
The function $W_\sigma:\bR\to \bR$ defined by \eqref{eq:g} satisfies
\begin{equation}\label{eq:gpos} 
	W_\sigma(v) \geq 0 , \quad \forall v\in \bR; \qquad   W_\sigma^{-1}(0) =  \{0,\theta\};
\end{equation}
and
\begin{equation}\label{eq:gmin}
	W_\sigma(v)\leq \frac 1 {2\sigma}   \min \left\{ v^2,  (v-\theta)^2\right\}.
\end{equation}
Furthermore, \ref{hyp:fGrowthCondition}   implies there exists $\nu > 0$ such that
\begin{equation}\label{eq:dg2} 
	W_\sigma(v) \geq \nu v^2 \quad\mbox{ for $v$ large enough}.
\end{equation}
\end{lemma}
The $\Gamma$-convergence of $\sF_\eps$ w.r.t.\ strong $L^1(\Omega)$-convergence toward the perimeter functional is thus a classical result (see  \cite{MM77,M87,S88}) and we will use this below to prove Theorem \ref{thm:Gamma}.

\medskip

\subsection{Proof of the \texorpdfstring{$\bm{\liminf}$}{lim inf} property (Theorem \ref{thm:Gamma}-(i))}\label{sec:liminf}
The proof of the $\liminf$ property (Theorem \ref{thm:Gamma}-(i)) follows from \eqref{eq:lowerbound} and the usual argument for Modica-Mortola's $\Gamma$-convergence of $\sF_\eps$, with a few adjustments to account for the indirect mass constraint on $\phi$.

We now consider a positive sequence $\varepsilon_n \to 0$ and a sequence $\pseq{\phi^{\eps_n}}_n$ that converges to $\phi$ in $L^1(\Omega)$.
If $\liminf_{n\to \infty} \sJ_{\eps_n} (\phi^{\eps_n})=+\infty$, then the result is obviously true, so we assume that \\$\liminf_{n\to \infty} \sJ_{\eps_n} (\phi^{\eps_n})<+\infty$ and consider a (not relabeled) subsequence $\pseq{\phi^{\eps_n}}_n$ that achieves the limit inferior, $\lim_{n\to\infty}  \sJ_{\eps_n} (\phi^{\eps_n}) = \liminf_{n\to \infty} \sJ_{\eps_n} (\phi^{\eps_n})$, and has bounded energy, $\sup_n \sJ_{\eps_n} (\phi^{\eps_n})\leq C < +\infty$.
\medskip

First, we note that \eqref{eq:Jepsbis} implies
\begin{equation}\label{eq:int1}
	\left| \int_{\Omega}\sigma \phi^\varepsilon \, dx -1 \right| \leq   \int_{\Omega} | \sigma \phi^\varepsilon-\rho_{\phi^\varepsilon} | \, dx \leq |\Omega|^{1/2}\left( \int_{\Omega} | \sigma \phi^\varepsilon-\rho_{\phi^\varepsilon} |^2 \, dx\right)^{1/2} \leq C \eps^{1/2} \sJ_\eps(\phi^\varepsilon)^{1/2}.
\end{equation}
Hence, the limit $\phi$ satisfies $\int_\Omega \phi(x)\, dx =1/\sigma$.

\medskip

From the strong $L^1$-convergence, we can further assume (up to another not relabeled subsequence) that $\phi^{\eps_n}\to\phi$ a.e.\ in $\Omega$.
On the one hand, $ \int_\Omega W_\sigma(\sigma\phi^{\eps_n})  \, dx \leq \varepsilon_n \sF_{\varepsilon_n}(\phi^{\varepsilon_n})\le \eps_n  \sJ_{\eps_n} (\phi^{\eps_n}) \le C \varepsilon_n \to 0$. By nonnegativity of $W_\sigma$, we see $W_\sigma(\sigma\phi^{\varepsilon_n}) \to 0$ in $L^1(\Omega)$. On the other hand, continuity of $W_\sigma$ gives $W_\sigma(\sigma\phi^{\varepsilon_n}) \to W_\sigma(\sigma\phi)$ a.e., so combining this with Fatou's lemma and the previous convergence shows $\int_\Omega W_\sigma(\sigma\phi) \,dx = 0$. Nonnegativity of $W_\sigma$ then gives $W_\sigma(\sigma\phi) = 0$ a.e., so we deduce
that $\phi = \frac\theta \sigma \chi_E$ for some set $E\subset \Omega$.

It remains to show that $\phi$ is of bounded variation. Towards this end, we introduce the auxiliary function $F_\sigma$ that satisfies 
\begin{equation}\label{eq:F} 
F_\sigma'(v)  =\frac{1}{\sigma}  \sqrt{2W_\sigma(v)}\, , \quad v\in(0,\theta), \quad F_\sigma(0)=0, 
\end{equation}
which we extend from $[0,\theta)$ to $\bR_+$ by a constant by setting $F_\sigma(v) :=\frac{1}{\sigma} \int_0^{\theta} \sqrt{2W_\sigma(s)}  \, ds = \frac{\gamma}{\sigma}\theta$ for $v\geq \theta$. (The definition of $\gamma$, \eqref{eq:sigma}, implies $\gamma = \frac \sigma \theta F_\sigma(\theta)$).
We point out that \eqref{eq:gmin} implies
\begin{equation}\label{eq:FLipschitz}
F_\sigma'(v)
\leq \frac{1}{\sigma^{3/2}} (\min\{ v, \theta-v\})_+ \le \frac{\theta}{2\sigma^{3/2}}\qquad \forall v\geq 0,
\end{equation}
so $F_\sigma \colon \bR_+ \to \bR_+$ is a bounded Lipschitz function, 
which implies $\pseq{F_\sigma (\sigma \phi^{\eps_n})}_n$ converges strongly in $L^1(\Omega)$ to $F_\sigma(\sigma\phi) = F_\sigma(\theta \chi_{E}) = F_\sigma(\theta )\chi_E =  \gamma \frac{\theta}{\sigma} \chi_E = \gamma\phi$.

We have (using \eqref{eq:F}, Young's inequality,  and \eqref{eq:lowerbound})
\begin{equation}\label{eq:BVbound} 
\begin{aligned}
	\int_\Omega |\na F_\sigma(\sigma\phi^{\eps_n})| \,dx &=  \int_\Omega \sqrt{2 W_\sigma(\sigma \phi^{\eps_n})} |\na \phi^{\eps_n}| \,dx \\
	&\leq  \frac {1} {\eps_n} \int_\Omega W_\sigma(\sigma\phi^{\eps_n})  \, dx  + \eps_n \int_\Omega \frac1  2 |\na \phi^{\eps_n}|^2\, dx
	= \sF_{\varepsilon_n}(\phi^{\varepsilon_n}) \leq 
	\sJ_{\eps_n}(\phi^{\eps_n}).
\end{aligned}
\end{equation}
The lower semicontinuity of the $BV$-seminorm gives
$$
C \ge \liminf_{n\to\infty}  \sJ_{\eps_n} (\phi^{\eps_n}) \geq
\liminf_{n\to\infty}   \int_\Omega |\na F_\sigma(\sigma\phi^{\eps_n})| 
\geq  \int_\Omega |\na (\gamma \phi)| = \sJ_0(\phi)
$$
and concludes the proof of the $\liminf$ property.

\medskip

\subsection{Proof of the \texorpdfstring{$\bm{\limsup}$}{lim sup} property (Theorem \ref{thm:Gamma}-(ii))}\label{sec:limsup}
The proof uses similar arguments as the proof of the corresponding result in \cite[Section 3]{M23}. We provide the details here for the reader's convenience.
The following equality, which is easily derived from \eqref{eq:g} plays an important role in the proof:
\begin{equation}\label{eq:leg}
W_\sigma(\sigma v) = \frac{\sigma}{2}v^2 - f^*(v-a), 
\end{equation}
where $f^*$ denotes the Legendre transform of $f$, defined by 
$$ f^*(v) :=\sup_{u\in \bR} \left( u v - f(u)\right) $$
(with the understanding that $f(u)=+\infty$ if $u<0$).
In particular, when $f$ is given by \eqref{eq:porousMediumf}, we have $f^*(v) = c_m v_+^{\frac m{m-1}}$.
In general, our assumptions on $f$ imply that $f^*$ is a convex function in $C^{1}(\bR) \cap C^2_\mathrm{loc}(\bR\setminus \{0\})$  satisfying ${f^*}'(v)=0$ for $v\leq 0$ and ${f^*}'(v)= (f')^{-1}(v)>0$ for all $v>0$ (see \cite[Section 3]{M23} for details).
In addition, we have the following important property:
\begin{equation}\label{eq:eqg} 
W_\sigma(\sigma v) = W(u) + \frac 1 {2\sigma} (u-\sigma v) \Longleftrightarrow u = {f^*}'(v -a)
\end{equation}
which follows from the fact that $ f^*(v)=  u v - f(u)$ iff $u =  {f^*}'(v)$.

\medskip

If $\sJ_0(\phi)=+\infty$, the construction of the recovery sequence is trivial, so we  can  assume that 
$\phi \in BV(\Omega; \frac{\theta}{\sigma} \})$
and $\int_\Omega \phi\, dx =\frac 1 \sigma$.
The first step is to use the well-known $\Gamma$-convergence of the Modica-Mortola functional \eqref{eq:ModMortF} to $\sJ_0$ 
(see \cite{S88,FT89,Leoni}), which implies the existence 
of a sequence $\seq{\phi^{\eps_n}}_n\subset H^1(\Omega) $ such that $\phi^{\eps_n}\to \phi $ in $L^1(\Omega)$ and 
$$ 
\limsup_{n\to\infty} \sF_{\eps_n}(\phi^{\eps_n}) \leq \sJ_0(\phi).
$$
In the usual construction of this recovery sequence, 
(see for instance \cite{Leoni}), 
one actually constructs a family of sequences $\pseq{\seq{\phi^{\eps_n}_t}_n}_t$ with $t\in[0,1]$ satisfying the conditions above; that is:
\begin{equation}\label{eq:fbhkj}
\phi^{\eps_n}_t\xrightarrow{n \to \infty} \phi  \mbox{ in } L^1(\Omega),\quad \mbox{ and } 
\quad
\limsup_{n \to\infty}  \sF_{\eps_n}(\phi^{\eps_n}_t)  \leq \sJ_0(\phi) \qquad \forall t\in [0,1]
\end{equation}
with $\phi^{\eps_n}_0\leq \phi$,  $\phi^{\eps_n}_1 \geq \phi$ and $t\mapsto\phi^{\eps_n}_t$ continuous in $L^1(\Omega)$.
The classical recovery sequence is then obtained by choosing $t_{\eps_n}$  so that  $\int_\Omega\phi^{\eps_n}_{t_{\eps_n}}\, dx= 1/\sigma$.
For our proof, we will make a slightly different choice of $t_{\eps_n}$:
We claim that one can choose $t_{\eps_n}\in[0,1]$ so that 
\begin{equation}\label{eq:intphi}
 \int_\Omega  \phi^{\eps_n}_{t_{\eps_n}} - W_\sigma'(\sigma \phi^{\eps_n}_{t_{\eps_n}})\,dx = 1/\sigma.
 \end{equation}
Indeed, the function $v\mapsto   v -   W_\sigma'(\sigma v)$ is  nondecreasing 
(since \eqref{eq:leg} implies that it is the derivative of the convex function $v \mapsto \frac 1 \sigma f^*(v-a)$)
and so 
$$ \phi^{\eps_n}_{0} -  W_\sigma'(\sigma \phi^{\eps_n}_{0}) \leq \phi -   W_\sigma'(\sigma \phi) =\phi,  \qquad  \phi^{\eps_n}_{1} -  W_\sigma'(\sigma \phi^{\eps_n}_{1}) \geq \phi  -  W_\sigma'(\sigma \phi) = \phi$$
(here we used the fact that $W_\sigma$ is $C^1$ and has minimums at $0$ and $ \theta $ together with the fact that $\sigma \phi$ only takes value $0$ and $ \theta$).
A continuity argument implies that there exists $t(\varepsilon_n)\in [0,1]$ such that \eqref{eq:intphi} holds.
We then denote $\phi^{\eps_n} :=\phi^{\eps_n}_{t(\varepsilon_n)}$ and 
$$\rho^{\eps_n}(x) :=\sigma (\phi^{\eps_n}(x) - W_\sigma'(\sigma \phi^{\eps_n}(x))) =  {f^*}'(\phi^{\eps_n}-a).$$

The fact that ${f^*}' (v)\geq 0$ 
 and \eqref{eq:intphi} give $\rho^{\eps_n} \in \cP_\ac(\Omega)$
and  \eqref{eq:eqg}  implies $W_\sigma(\sigma \phi^{\eps_n}) = W(\rho^{\eps_n}) + \frac 1 {2\sigma} (\rho^{\eps_n}-\sigma \phi^{\eps_n})$ and so the definition of $\sJ_{\eps_n}$ (see \eqref{eq:J0}) gives
\begin{align*}
\sF_{\varepsilon_n}(\phi^{\eps_n}) & =  
\frac 1 {\eps_n} \int_\Omega W_\sigma(\sigma \phi^{\eps_n})  \, dx  + {\eps_n} \int_\Omega \frac1  2 |\na \phi^{\eps_n}|^2\, dx \\
& = 
 \frac {1} {\eps_n} \int_\Omega W(\rho^{\eps_n})  + \frac 1 {2\sigma}  (\rho^{\eps_n}-\sigma \phi ^{\eps_n})^2\, dx +{\eps_n}  \int_\Omega \frac1  2 |\na \phi^{\eps_n} |^2\, dx 
\\
&  \geq \sJ_{\eps_n}(\phi^{\eps_n}).
  \end{align*} 
$$ \limsup_{n \to \infty}\sJ_{\eps_n}(\phi^{\eps_n} ) \leq  \limsup_{n\to \infty}\sF_{\eps_n}(\phi^{\eps_n})  \leq \sJ_0(\rho),$$
which gives the $\limsup$ property since $\phi^{\eps_n}\to \phi$ in $L^1(\Omega)$.

\medskip

\section{Convergence of the density \texorpdfstring{$\rho^\eps$}{\textrhoᵋ} and chemoattractant \texorpdfstring{$\phi^\eps$}{ϕᵋ}}
\label{sec:conv}
We now turn our attention to Theorem \ref{thm:main}. In this section, we prove the first part of this theorem, namely the strong convergence of $\phi^{\eps_n}$ and $ \rho^{\eps_n}$ to (multiples of) the characteristic function of a set of finite perimeter.
We recall that $\phi^{\eps}$, the solution of \eqref{eq:phi}, satisfies the energy dissipation inequality
\begin{equation}\label{eq:energy}
	\sJ_{\eps}( \phi^{\eps}(t)) + \int_0^t \sD_\eps(s)\,ds\leq \sJ_{\eps}( \phi_{\init}^\varepsilon), \qquad \sD_\eps(t) := \int_\Omega \eps |\pa_t \phi^{\eps}(t,x)|^2 \, dx,
\end{equation}
so well-preparedness of the initial data, i.e., $\sup_{\varepsilon \in (0,\varepsilon_0)} \sJ_\varepsilon(\phi_\init^\varepsilon) \le C < +\infty$, implies that both the energy and the dissipation are uniformly bounded in $t \in [0,T]$ for arbitrarily small $\varepsilon\in (0,\varepsilon_0)$:
\begin{equation} \label{eq:energyDissipationBound}
	\sup_{\varepsilon \in (0,\varepsilon_0)} \sup_{t \in [0,T]} \left\{ \sJ_{\eps}( \phi^{\eps}(t)) + \int_0^t \sD_\eps(s)\,ds \right\} \le C.
\end{equation}

There are two essential consequences of this. First, the definition of $\sJ_\varepsilon$ \eqref{eq:Jepsbis} implies
\begin{equation} \label{eq:L2DistanceBound}
	\int_{\Omega} | \sigma \phi^\varepsilon(t)-\rho_{\phi^\varepsilon(t)} | \, dx \leq |\Omega|^{1/2}\left( \int_{\Omega} | \sigma \phi^\varepsilon(t) - \rho_{\phi^\varepsilon(t)} |^2 \, dx\right)^{1/2} \leq C \eps^{1/2} \sJ_\eps(\phi^\varepsilon(t))^{1/2} \le C \eps^{1/2},
\end{equation}
so $\sigma\phi^\varepsilon$ and $\rho_{\phi^\varepsilon}$ have the same  limit as $\varepsilon \to 0^+$.
In addition, \eqref{eq:lowerbound}-\eqref{eq:ModMortF} gives
\begin{equation} \label{eq:doubleWellPotentialBound}
	\int_\Omega W_\sigma(\sigma \phi^\varepsilon(t)) \,dx \le  \varepsilon\sF_\varepsilon(\phi^\varepsilon(t)) \le \varepsilon\sJ_\varepsilon(\phi^\varepsilon(t)) \le C \varepsilon,
\end{equation}
so the range of any sufficiently strong limit $\sigma\phi$ of $\pseq{\sigma \phi^\varepsilon}_\varepsilon$ must be contained in the wells of $W_\sigma$. That is, $\sigma\phi$ ought to be (a multiple of) a characteristic function.

These estimates will be key to proving the following proposition that establishes Part (i) (phase separation) of Theorem \ref{thm:main}.
\begin{proposition}\label{prop:conv}
	Let $\phi^{\eps_n}$ be as in Theorem \ref{thm:main} and define the auxiliary function 
	$$ \psi^{\varepsilon_n}(t,x) :=F_\sigma(\sigma \phi^{\eps_n}(t,x)),$$
	where $F_\sigma \colon \bR_+ \to \bR_+ $ is defined by \eqref{eq:F}.
	There exists $\phi\in C^{0,s}([0,T]; L^1(\Omega)) \cap  BV((0,T)\times\Omega) \cap L^\infty(0,T ; BV(\Omega))$ for all $s \in [0,\frac{1}{2})$ 
	such that, up to a subsequence, 
	\begin{gather*} 
		\psi^{\varepsilon_n} \xrightarrow{n\to\infty} \gamma \phi \quad \text{ strongly in } C^{0,s}([0,T]; L^1(\Omega)),  \\
		\phi^{\eps_n} \xrightarrow{n\to\infty} \phi, \qquad \rho_{\phi^{\eps_n}} \xrightarrow{n\to\infty} \sigma \phi \quad \mbox{ strongly in $L^\infty([0,T];L^q (\Omega))$}, \quad \forall q\in[1,\frac{d}{d-1})
	\end{gather*}
	with $\gamma  = \frac \sigma \theta F_\sigma(\theta)$. Furthermore, for all $t \in [0,T]$, there exists a set  $E(t)\subset\Omega$ such that $\phi(t)=\frac \theta \sigma\chi_{E(t)}$ and \begin{equation}\label{eq:Et}
		|E(t)| = \frac{1}{\theta},\, \qquad P(E(t),\Omega)<+\infty \qquad \forall t \in [0,T].
	\end{equation}
\end{proposition}
Above, $P(E(t),\Omega)$ denotes the perimeter of the set $E(t)$ in $\Omega$,  defined  by
$$
P(E(t),\Omega) :=\int_\Omega |\nabla \chi_{E(t)}|.
$$

\begin{proof}
		Lipschitz continuity of $F_\sigma$ gives 
	\[
	\abs{\psi^{\varepsilon_n}(t)} \le C  \abs{\sigma\phi^{\varepsilon_n}(t)} \le C \big( \rho_{\phi^{\varepsilon_n}(t)} + \abs{\sigma\phi^{\varepsilon_n}(t) - \rho_{\phi^{\varepsilon_n}(t)}} \big).
	\]
	Using that $\int_\Omega \rho_{\phi^{\varepsilon_n}(t)}(x) \,dx = 1$ for all $t \in [0,T]$ and $n \in \bN$ and the calculation \eqref{eq:L2DistanceBound}, we deduce an  $L^\infty(0,T;L^1(\Omega))$-bound on $\psi^{\varepsilon_n}$.
	
	Next, inequality \eqref{eq:BVbound} gives
	$$ \int_\Omega |\na \psi^{\varepsilon_n}(t) | \,dx \leq  \sJ_{\eps_n}(\phi^{\eps_n}(t))\leq C.
	$$
so $ \psi^{\varepsilon_n}$ is bounded in $L^\infty(0,T;BV(\Omega))$.

	
	
	Furthermore, using the definition of $\psi^{\varepsilon_n}$, Cauchy-Schwarz, the uniform-in-$t$-and-$n$ bound on the energy from the double-well potential \eqref{eq:doubleWellPotentialBound}, and the uniform-in-$n$ bound on the dissipation \eqref{eq:energyDissipationBound}:
	\begin{align*}
		\int_0^T \left(\int_\Omega |\pa_t \psi^{\varepsilon_n}|\, dx\right)^2\,dt 
		&=\int_0^T\left( \int_\Omega \sqrt{2 W_\sigma(\sigma\phi^{\eps_n}) }  |\pa_t \phi^{\eps_n}|\, dx\right)^2\,dt  \\
		&\leq   \int_0^T 
		\int_\Omega \frac{2}{\eps_n}  W_\sigma(\sigma\phi^{\eps_n} )  \, dx  \int_\Omega \eps_n  |\pa_t \phi^{\eps_n}|^2\, dx\,dt \\
		&\leq  2 C \int_0^T 
		\int_\Omega \eps_n  |\pa_t \phi^{\eps_n}|^2\, dx\,dt \\
		&\leq 2C^2.
	\end{align*}
	Thus, $\pa_t \psi^{\varepsilon_n}$ is bounded in $L^2(0,T; L^1(\Omega))$. 
		
Since $BV(\Omega)$ is compactly embedded into $L^q(\Omega)$ for all $q < \frac{d}{d-1}$,
a refined  
	Aubin-Lions lemma \cite[Theorem 1.1]{AubinLions} thus implies  that $\pseq{\psi^{\varepsilon_n}}_n$ is relatively compact in $C^{0,s}([0,T];L^1(\Omega))$ for  $0 \le s < \frac{1}{2}$ and $L^\infty([0,T];L^q(\Omega))$ for  $1 \le q< \frac d{d-1}$.
	Up to a subsequence, we can thus assume that $\psi^{\varepsilon_n}$ converges to $\phi$ strongly in $C([0,T];L^q(\Omega))$. 
	The lower semicontinuity of the $BV(\Omega)$-seminorm (resp. $BV((0,T)\times\Omega)$-seminorm) gives
	$\phi \in L^\infty(0,T; BV(\Omega))$  (resp. $\phi \in BV((0,T)\times\Omega)$).
	
	
	\medskip
	
	Next, we note that the function $\bR_+ \ni v \mapsto F_\sigma(v)-\frac{\gamma}{\sigma} v$ vanishes when $v=0 $ and $v=\theta$ (see the discussion after \eqref{eq:F}), which are also the zeroes of the double-well potential $W_\sigma$.
	So given a sufficiently small $\delta > 0$ and neighborhood $V_\delta :=[0,\delta) \cup (\theta-\frac{\delta}{2},\,\theta + \frac{\delta}{2})$,  
	the continuity of $F_\sigma$ and $W_\sigma$ implies
	$$ 
	| F_\sigma (v) -\frac{\gamma}{\sigma} v|^2 \leq C_\delta W_\sigma(v)\qquad \mbox{ for } v \notin V_\delta
	$$ 
	for some constant $C_\delta$. More precisely, \eqref{eq:dg2} and the fact that $F_\sigma$ is bounded imply this inequality for large $v \notin V_\delta$, and then the continuity together with a compactness argument imply the inequality for small $v \notin V_\delta$.
	Since $F_\sigma$ is Lipschitz, we also have
	$$ | F_\sigma(v)-\frac{\gamma}{\sigma} v | \leq C\delta\qquad \mbox{ for } v\in V_\delta.$$
	We deduce (using \eqref{eq:doubleWellPotentialBound}):
	\begin{align*}
		\int_\Omega   |F_\sigma (\sigma \phi^{\eps_n}(t)) -\gamma \phi^{\eps_n}(t)|^2 \, dx 
		& \leq \int_{\{\phi^{\eps_n}(t)\in V_\delta\}}   C\delta^2 \, dx  + C_\delta \int_{\{\phi^{\eps_n}(t)\notin V_\delta\}}  W_\sigma(\sigma \phi^{\eps_n}(t))\, dx\\
		& \leq C|\Omega|  \delta^2 + C _\delta \eps_n   \sJ_\eps(\phi^{\eps_n}(t))\\
		& \leq C|\Omega|  \delta^2 + C _\delta \eps_n    
	\end{align*}
	hence
	$$\limsup_{n\to\infty} \| F_\sigma(\sigma \phi^{\eps_n} ) -\gamma \phi^{\eps_n}\|_{C([0,T];L^2(\Omega))} \leq C \delta. 
	$$
	Since the left-hand side is independent of $\delta$ and this holds for all $\delta>0$, we must have
	\begin{equation}\label{eq:ghj2}
		\limsup_{n\to\infty} \| F_\sigma(\sigma \phi^{\eps_n} ) -\gamma \phi^{\eps_n}\|_{C([0,T];L^2(\Omega))} =0.
	\end{equation}
	It follows that the strong convergence of $\psi^{\varepsilon_n} :=F_\sigma(\sigma \phi^{\eps_n} )$, 
	implies the same strong convergence of $\phi^{\eps_n}$ to $\phi$. Up to another subsequence, we can thus assume $\phi^{\varepsilon_n} \to \phi$ for each $t\in[0,T]$ a.e.\ in $\Omega$, and then continuity of $W_\sigma$ gives $W_\sigma(\sigma\phi^{\varepsilon_n}(t)) \to W_\sigma(\sigma\phi(t))$ for each $t\in[0,T]$ a.e.\ in $\Omega$, too. 
	Nonnegativity of $W_\sigma$, Fatou's lemma, and \eqref{eq:doubleWellPotentialBound} give for each $t \in[0,T]$
	\[
	\int_\Omega W_\sigma (\sigma \phi(t))\, dx \le \liminf_{n\to\infty} \int_\Omega W_\sigma(\sigma \phi^{\eps_n}(t))\, dx \le \liminf_{n\to\infty} \big[  \varepsilon_n \sJ_{\varepsilon_n}(\phi^{\varepsilon_n}(t)) \big] 
	\le \liminf_{n\to\infty}\big[ C\varepsilon_n \big] = 0.
	\]
	Since $W_\sigma(\sigma\phi(t))$ is nonnegative, we deduce $W_\sigma(\sigma\phi(t)) = 0$ a.e.\ in $\Omega$, and so $\phi(t) = \frac{\theta}{\sigma}\chi_{E(t)}$ for some set $E(t)\subset \Omega$. We already know that $\phi(t) \in BV(\Omega)$, which means that the  set $E(t)$ has finite perimeter.

	Finally, \eqref{eq:L2DistanceBound} implies 
	\[
	\limsup_{n\to\infty} \norm{\sigma\phi^{\varepsilon_n} - \rho_{\phi^{\varepsilon_n}}}_{L^\infty(0,T; L^2(\Omega))}  \le C\limsup_{n \to \infty} \big[ \varepsilon_n \sup_{t \in [0,T]}\sJ_{\varepsilon_n}(\phi^{\varepsilon_n}(t)) \big]^{1/2} \le C \limsup_{n\to\infty} \varepsilon_n^{1/2} = 0.
	\]
	Thus, the strong convergence of $\phi^{\eps_n}$ implies that of $\rho_{\phi^{\eps_n}} $. 
\end{proof}

\section{The Velocity}\label{sec:V}
The next two sections are devoted to the proof of Theorem \ref{thm:main}-(ii), which identifies the evolution equation satisfied by the limits obtained from the phase separation result, Theorem \ref{thm:main}-(i). 
Thus, from now on, we assume that $\pseq{\eps_n}_n \subset (0,\infty)$ is a sequence converging to zero such that  the conclusion of Proposition \ref{prop:conv} and the assumption \eqref{eq:EA} of convergence of the energy hold.

Classically (see \cite[Lemma 2.11]{Laux3} and \cite[Lemma 1]{LM}), this energy convergence assumption  implies that, in the limit $\varepsilon \to 0^+$, there is equality in both Young's inequality $\frac{1}{\varepsilon}W_\sigma(\sigma \phi^\varepsilon) + \frac{\varepsilon}{2}\abs{\nabla\phi^\varepsilon}^2 \ge \abs{\nabla F_\sigma(\sigma\phi^\varepsilon)}$ (see \eqref{eq:L2diff}) and in the inequality $W(\rho_{\phi^\varepsilon}) + \frac{1}{2\sigma}(\rho_{\phi^\varepsilon} - \sigma\phi^\varepsilon)^2 \ge W_\sigma(\sigma\phi^\varepsilon)$ (see \eqref{eq:Wg}). As a consequence, at a.e.\ fixed time, the contributions to the energy $\sF_\varepsilon$ converge weakly-$\ast$ as Radon measures on $\cl\Omega$ to (an equal proportion of) the limiting contributor to the limiting energy $\sJ_0$ (see \eqref{eq:EA1}).
This important fact is made precise in the following lemma:
\begin{lemma}[Equipartition of the energy]\label{lem:cv}
	Under assumption \eqref{eq:EA}, and up to another subsequence, we have, for a.e. $t\in [0,T]$ and any test function $\zeta \in C(\cl\Omega)$,
	\begin{equation}\label{eq:EA1}
		\begin{aligned}
			\frac{\gamma}{2} \int_\Omega \zeta(x) |\na \phi(t,x)|
			& = \lim_{n\to \infty} \frac{1}{2}\int_\Omega | \na F_\sigma( \sigma \phi^{\eps_n}(t,x))|\zeta(x)\, dx \\
			&  = \lim_{n\to \infty} \int_\Omega \frac{\eps_n}{2}|\na \phi^{\eps_n}(t,x)|^2\zeta(x)\, dx \\
			& = \lim_{n\to \infty}\int_\Omega \frac{1}{\eps_n}W_\sigma(\sigma \phi^{\eps_n}(t,x))\zeta(x)\, dx.
		\end{aligned}
	\end{equation}
	Furthermore, we have 
	\begin{equation}\label{eq:L2diff} 
		\lim_{n \to \infty} \Big(\eps_n^{1/2 }|\na \phi^{\eps_n}| - \frac{1}{\eps_n^{1/2}} \sqrt{2 W_\sigma(\sigma \phi^{\eps_n})}\Big) = 0\quad  \mbox{ strongly in } L^2((0,T)\times\Omega)
	\end{equation}
	and 
	\begin{equation}\label{eq:Wg}
		\lim_{n \to \infty} \Big(\frac 1 {\eps_n} W(\rho_{\phi^{\eps_n}}) + \frac 1 {2\sigma\varepsilon_n}  (\rho_{\phi^{\eps_n}}-\sigma \phi ^{\eps_n})^2 - \frac{1}{\eps_n}W_\sigma(\sigma \phi^{\eps_n})\Big) = 0 \quad \mbox{ strongly in } L^1((0,T)\times\Omega).
	\end{equation}
\end{lemma}

\begin{proof}
	
	We first upgrade the convergence of the energy \eqref{eq:EA} to $L^1(0,T)$ and a.e. convergence. We write: 
	\[
	\abs{\sJ_{\eps_n}(\phi^{\eps_n}(t)) - \sJ_0(\phi(t))} = 2(\sJ_0(\phi(t)) -\sJ_{\eps_n}(\phi^{\eps_n}(t)) )_+ + \big[\sJ_{\eps_n}(\phi^{\eps_n}(t)) - \sJ_0(\phi(t))\big].
	\]
	After an integration over $t \in (0,T)$, the second term on the right-hand side vanishes as $n \to \infty$ by the energy convergence assumption \eqref{eq:EA}, so we only need to address the first term. Since $\phi^{\eps_n}\to \phi$ strongly in $C([0,T]; L^1(\Omega))$,  we have that $\phi^{\eps_n}(t) \to \phi(t)$ in $L^1(\Omega)$ for each $t \in [0,T]$.
	Then, the $\liminf$ property of $\Gamma$-convergence (Theorem \ref{thm:Gamma}-(i)) implies that 
	$\lim_{n\to\infty} (\sJ_0(\phi(t)) -\sJ_{\eps_n}(\phi^{\eps_n}(t)) )_+ =0$ for each $t \in [0,T]$. Since $\sJ_{\varepsilon_n}$ is nonnegative, we have for each $t$ the upper bound $(\sJ_0(\phi(t)) -\sJ_{\eps_n}(\phi^{\eps_n}(t)) )_+ \le \sJ_0(\phi(t))$.
	Thus, Lebesgue dominated convergence gives $(\sJ_0(\phi(\cdot)) -\sJ_{\eps_n}(\phi^{\eps_n}(\cdot)) )_+ \to 0$ in $L^1(0,T)$ and we deduce
	$$ \lim_{n\to\infty} \int_0^T \abs{ \sJ_{\eps_n}(\phi^{\eps_n}(t)) - \sJ_0(\phi(t))}\, dt =0.$$
	Up to another subsequence, we can also assume that 
	\begin{equation}\label{eq:EAt}
		\lim_{n\to\infty}\sJ_{\eps_n}(\phi^{\eps_n}(t)) = \sJ_0(\phi(t)) \mbox{ for a.e.\ } t\in[0,T].
	\end{equation}
	
	We now address the first limit in \eqref{eq:EA1}, which is a classical result associated to Modica-Mortola functionals, like $\sF_\varepsilon$. Recall for each $t \in [0,T]$ that $F_\sigma(\sigma \phi^{\eps_n}(t,\cdot)) \to \gamma \phi(t,\cdot)$ in $L^1(\Omega)$ since $ F_\sigma $ is Lipschitz, \eqref{eq:FLipschitz}. Thus, the lower semicontinuity of the $BV(\Omega)$-seminorm,	\eqref{eq:BVbound} and \eqref{eq:EAt} imply, for a.e.\ $t \in [0,T]$, that 
	\begin{align*}
		\sJ_0(\phi(t)) \leq  \liminf_{n\to\infty}  \int_\Omega  | \na F_\sigma(\sigma \phi^{\eps_n}(t,x))|\, dx &\leq  \limsup_{n\to\infty}  \int_\Omega  | \na F_\sigma(\sigma \phi^{\eps_n}(t,x))|\, dx \\
		&\leq   \lim_{n\to\infty} \sJ_{\eps_n}(\phi^{\eps_n}(t)) = \sJ_0(\phi(t)).
	\end{align*}
	We thus have equality in these inequalities, in particular:
	\begin{equation}\label{eq:lim1}
		\lim_{n\to\infty}\int_\Omega  | \na F_\sigma(\sigma  \phi^{\eps_n}(t,x))|\, dx = \sJ_0(\phi(t))  :=\int_\Omega |\na (\gamma \phi(t,x))| \text{ for a.e.\ } t \in [0,T].
	\end{equation}
This convergence of the total variation 
implies the first limit of \eqref{eq:EA1} (see \cite[Proposition 3.15]{AFP_BV_00}). 
	
	In fact, \eqref{eq:lim1} and Lebesgue dominated convergence imply $L^1(0,T)$-convergence of the variations to $\sJ_0(\phi(\cdot))$. Since we have already shown $\lim_{n\to\infty}\sJ_{\varepsilon_n}(\phi^{\varepsilon_n}(\cdot)) = \sJ_0(\phi(\cdot))$ for a.e.\ $t \in [0,T]$ and in $L^1(0,T)$, we further see that the deviation $z^{\varepsilon}$ between $\sJ_{\varepsilon}(\phi^\varepsilon(\cdot))$ and the optimal lower bound on the Modica-Mortola functional $\sF_\varepsilon(\phi^\varepsilon(\cdot))$ vanishes in the limit:
	\begin{equation}\label{eq:zeps}
		\begin{gathered}
			z^{\eps_n}(t) :=\sJ_{\eps_n}(\phi^{\eps_n}(t)) - \int_\Omega  | \na F_\sigma(\sigma \phi^{\eps_n}(t,x))|\, dx \ge  \sF_{\eps_n}(\phi^{\eps_n}(t)) - \int_\Omega  | \na F_\sigma(\sigma \phi^{\eps_n}(t,x))|\, dx \geq 0, \\
			\lim_{n\to\infty} z^{\eps_n} = 0 \qquad \mbox{ a.e.\ } t \in [0,T] \text{ and strongly in } L^1(0,T).
		\end{gathered}
	\end{equation}

Next, we denote 
	\[
	u_n :=\frac{\eps_n}{2}|\na \phi^{\eps_n}|^2,	\quad 
	v_n :=\frac{1}{\eps_n}  W_\sigma(\sigma \phi^{\eps_n}), \quad 
	w_n :=\frac 1 {\eps_n} W(\rho_{\phi^{\eps_n}}) + \frac 1 {2\sigma\varepsilon_n}  (\rho_{\phi^{\eps_n}}-\sigma \phi ^{\eps_n})^2.
	\]
	In this notation, the energies take the form $\sF_{\varepsilon_n}(\phi^{\varepsilon_n}) = \int_\Omega v_n + u_n \,dx$ and $\sJ_{\varepsilon_n}(\phi^{\varepsilon_n}) = \int_\Omega w_n + u_n \,dx$. Furthermore, we recall that the fact that  $F_\sigma '(s) = \frac 1 \sigma \sqrt{2W_\sigma (s)}$ and 
	Young's inequality imply
	\begin{equation}\label{eq:young1}
	| \na F_\sigma(\sigma \phi^{\eps_n})| =  2 \sqrt{u_n v_n} \leq u_n+v_n
	\end{equation}
and since $\sF_\varepsilon \le \sJ_\varepsilon$, we see
	\begin{align*}
		0 \le \int_\Omega u_n(t) +  v_n(t)  -   | \na F_\sigma(\sigma \phi^{\eps_n}(t))|\, dx &= \sF_{\varepsilon_n}(\phi^{\varepsilon_n}(t)) - \int_\Omega | \na F_\sigma(\sigma \phi^{\eps_n}(t))|\, dx \leq z^{\eps_n}(t).
	\end{align*}
	So \eqref{eq:zeps} implies
	\begin{equation}\label{eq:lim2}
		\lim_{n\to\infty} \big(u_n +  v_n  -   | \na F_\sigma(\sigma \phi^{\eps_n})|\big) = 0\qquad \mbox{ strongly in } L^1((0,T)\times \Omega),
	\end{equation}
	which is asymptotic equality in Young's inequality \eqref{eq:young1}
	
	To be more precise, we can write
	\begin{equation}\label{eq:bjh}
		| \na F_\sigma(\sigma \phi^{\eps_n})|  = 2 \sqrt{u_n v_n} = u_n +  v_n -  (\sqrt{u_n}-\sqrt{v_n})^2=   u_n +w_n -   (\sqrt{u_n}-\sqrt{v_n})^2 - |w_n- v_n|,
	\end{equation}
where the last equality follows by noticing that the definition of $W_\sigma$ \eqref{eq:g} implies $w_n - v_n\geq 0$ pointwise. 
Rearranging \eqref{eq:bjh} gives 
	\begin{align*}
		0 \le \int _\Omega   \big(\sqrt{u_n(t)}-\sqrt{v_n(t)}\big)^2 + \abs{w_n(t)- v_n(t)}\, dx = \int_\Omega u_n(t) + w_n(t) - \abs{\nabla F_\sigma(\sigma\phi^{\varepsilon_n}(t))} \,dx  = z^{\eps_n}(t).
	\end{align*}
	Thus, 	
	\begin{equation}\label{eq:lim3}
		\int_\Omega (\sqrt{u_n(t)}-\sqrt{v_n(t)})^2 \, dx\leq z^{\eps_n}(t), \qquad\quad  \int_\Omega | w_n(t)-   v_n(t)|\, dx \leq z^{\eps_n}(t), 
	\end{equation}
	and so the combination of \eqref{eq:zeps} with the first inequality in \eqref{eq:lim3} gives \eqref{eq:L2diff} while \eqref{eq:zeps} and the second inequality in \eqref{eq:lim3} give \eqref{eq:Wg}. 
	
	It remains to show the last two limits in \eqref{eq:EA1}. Towards this end, the first inequality in \eqref{eq:lim3} shows the two energy densities in $\sF_\varepsilon$ are also equal as $\varepsilon \to 0^+$. Indeed, by introducing a difference of squares
	\begin{align*}
		\norm{u_n(t) - v_n(t)}_{L^2(\Omega)} \le \norm{\sqrt{u_n(t)} + \sqrt{v_n(t)}}_{L^2(\Omega)} \norm{\sqrt{u_n(t)} - \sqrt{v_n(t)}}_{L^2(\Omega)} 
		\le C \sqrt{\sF_{\varepsilon_n}(\phi^{\varepsilon_n}(t))} \sqrt{z^{\varepsilon_n}(t)}, 
	\end{align*}
	and since $\sF_{\varepsilon_n}(\phi^{\varepsilon_n}(t))$ is  bounded uniformly in $t \in [0,T]$ and $n\in\bN$, we see \eqref{eq:zeps} gives
	\[
	\lim_{n\to\infty} \big(u_n - v_n \big) = 0 \qquad \text{strongly in } L^2((0,T) \times \Omega).
	\]
	By combining this with \eqref{eq:lim2}, we obtain 
	\begin{equation} \label{eq:equipartitionModicaMortola}
	\lim_{n\to\infty} \big( u_n - \frac{1}{2} \abs{\nabla F_\sigma(\sigma\phi^{\varepsilon_n})}\big) = \lim_{n\to\infty} \big( v_n - \frac{1}{2} \abs{\nabla F_\sigma(\sigma\phi^{\varepsilon_n})}\big) = 0 \qquad \text{strongly in } L^1((0,T) \times \Omega).
	\end{equation}
	Thus, up to another subsequence, the above limits are for a.e.\ $t \in [0,T]$ strongly in $L^1(\Omega)$. Combining this with the first limit in \eqref{eq:EA1} provides the last two limits in \eqref{eq:EA1}.
\end{proof}

The existence of the normal velocity $V = V(t,x)$ now follows from the following lemma. This establishes the first part of Theorem \ref{thm:main}-(ii).
\begin{lemma}[Existence of the normal velocity $V$]\label{lem:vel}
	There exists a constant $C = C(\sup_{n \in \bN}\sJ_{\varepsilon_n}(\phi^{\varepsilon_n}_\init)) > 0$ such that, under the energy convergence assumption \eqref{eq:EA},  for all test functions $\zeta \in C^1_c((0,T) \times \Omega)$, we have
	$$ \left| \int_0^T\int_\Omega \phi(t,x) \pa_t\zeta(t,x)\, dx\, dt\right|\leq C\left(\int_0^T \int_\Omega \zeta(t,x)^2 |\na \phi| \, dt \right)^{1/2}.
	$$
	In particular, there exists $V\in L^2((0,T) \times \Omega, |\na\phi|\,dt)$ such that $  \pa_t \phi = V |\na \phi|\, dt$ as Radon measures on $[0,T] \times \cl\Omega$ and thus in the distributional sense of \eqref{eqLvelweak}. Moreover, for any $\zeta \in C^1_c((0,T) \times \Omega)$
	\begin{equation} \label{eq:convergenceOfVelocity}
		\lim_{n\to\infty} \int_0^T \int_\Omega \partial_t \psi^{\varepsilon_n} \zeta \,dx \,dt = \int_0^T \int_\Omega \gamma\zeta V\abs{\nabla\phi}\,dt.
	\end{equation}
\end{lemma}
\begin{proof}
	Recalling the definition $\psi^{\varepsilon_n}(t,x):=F_\sigma(\sigma\phi^{\eps_n}(t,x))$, we write (using the bound on the  dissipation \eqref{eq:energy}):
	\begin{align*}
		\left|\int_0^T\int_\Omega \psi^{\varepsilon_n} \pa_t\zeta\, dx\, dt\right|
		& = \left| \int_0^T\int_\Omega \sqrt{2W_\sigma(\sigma \phi^{\eps_n})}\pa_t\phi^{\eps_n} \zeta\, dx\, dt\right|\\	
		& \leq  \left(\int_0^T\int_\Omega \frac 1 {\eps_n} 2W_\sigma(\sigma\phi^{\eps_n}) \zeta^2 \, dx\,dt \right)^{1/2} \left( \int_0^T\int_\Omega \eps _n |\pa_t\phi^{\eps_n} |^2 \, dx\, dt\right)^{1/2}\\
		& \leq C  \left(\int_0^T\int_\Omega \frac 1 {\eps_n} 2W_\sigma(\sigma \phi^{\eps_n}) \zeta^2 \, dx\,dt \right)^{1/2}.	
	\end{align*}
Since $\psi^{\varepsilon_n}$ converges (strongly in $L^1((0,T)\times\Omega)$) to $\gamma \phi$,
and using the third equality in \eqref{eq:EA1} (together with Lebesgue dominated convergence theorem), we deduce
	\begin{equation}
		\left|\gamma \int_0^T\int_\Omega \phi \,\pa_t \zeta\, dx\, dt\right|
		\leq C \left(\gamma\int_0^T\int_\Omega   \zeta^2 |\na \phi| \,dt \right)^{1/2}.
	\end{equation}
	Thus, the measure $ \pd \phi t$ is a bounded linear functional on $L^2((0,T)\times\Omega, \abs{\nabla\phi}dt)$ and can be identified with a unique $V \in L^2((0,T) \times \Omega, \abs{\nabla \phi}dt)$ satisfying \eqref{eqLvelweak}.
	
	To prove \eqref{eq:convergenceOfVelocity}, we write:
	\[
	\lim_{n\to\infty} \int_0^T \int_\Omega \partial_t \psi^{\varepsilon_n} \zeta \,dx \,dt = -\lim_{n\to\infty} \int_0^T \int_\Omega \psi^{\varepsilon_n} \partial_t \zeta \,dx \,dt = -\int_0^T \int_\Omega \gamma \phi \partial_t \zeta \,dx\,dt = \int_0^T \int_\Omega \gamma \zeta \partial_t\phi
	\]
	and the result follows.
\end{proof}


\section{Volume-Preserving Mean-Curvature Flow}\label{sec:mc}
We now complete the proof of Theorem \ref{thm:main}-(ii) by establishing
\eqref{eq:Vxi}, which is the weak formulation of \eqref{eq:V}.
Many aspects of this derivation are classical, but we recall that unlike the usual Allen-Cahn equation, the right-hand side of \eqref{eq:phi} is nonlocal and the limiting equation preserves the volume.   
The proof of Proposition \ref{prop:LM} (existence of a Lagrange multiplier) is thus the most original part of this section.
\medskip

In order to derive \eqref{eq:Vxi}, we consider a vector field 
$\xi \in [C^1([0,T]\times\overline \Omega)]^d$ satisfying $\xi\cdot n=0$ on $\pa\Omega$, and (at a fixed but suppressed time $t \in [0,T]$) we multiply the $\phi$ equation \eqref{eq:phi} by $\eps \na \phi^\eps \cdot \xi$ to get:
$$
\int_\Omega  \eps \pa_t\phi^\eps \na \phi^\eps \cdot \xi\, dx
-  \int_\Omega  \eps \Delta \phi^\eps \na \phi^\eps \cdot \xi\, dx 
=  \int_\Omega  \frac{1}{\eps}(\rho_{\phi^\eps} - \sigma\phi^\eps) \na \phi^\eps \cdot \xi\, dx. 
$$
We note that the second term can be written (after a couple of integration by parts) as
$$
- \int_\Omega  \eps \Delta \phi^\eps \na \phi^\eps \cdot \xi\, dx 
=  \int_\Omega  \eps \na \phi^\eps\otimes \na \phi^\eps:D\xi\, dx 
-  \int_\Omega   \frac\eps 2  |\na \phi^\eps |^2 \div \xi\, dx
$$
and we rewrite the right-hand side as
$$
\int_\Omega  \frac{1}{\eps}(\rho_{\phi^\eps} - \sigma\phi^\eps) \na \phi^\eps \cdot \xi\, dx  
=
\frac{1}{\eps}  \int_\Omega  (\rho_{\phi^\eps} - \sigma\phi^\eps) \na \phi^\eps \cdot \xi - W_\sigma(\sigma\phi^\eps)\div\xi \, dx  
+  \frac{1}{\eps}  \int_\Omega W_\sigma (\sigma \phi^\eps) \div  \xi\, dx.  
$$
This rewriting is natural from the point-of-view of the classical Allen-Cahn equation. The first term on the right-hand side measures (in a weak sense) the deviation between our nonlocal reaction term in \eqref{eq:phi} and the standard local reaction term. The second term is what would arise from the standard local reaction term, and from Proposition \ref{lem:cv}, we already know at a.e.\ fixed $t\in[0,T]$ that  $\pseq{W_\sigma(\sigma\phi^\varepsilon(t))}_\varepsilon$ converges weakly-$\ast$ as Radon measures on $\cl\Omega$ to one-half of the limiting total variation. 

Putting these equalities together, we deduce
\begin{equation}\label{eq:eqre}
	\begin{aligned}
		\int_\Omega  \eps \pa_t\phi^\eps \na \phi^\eps \cdot \xi\, dx 
		& = - \int_\Omega  \eps \na \phi^\eps\otimes \na \phi^\eps:D\xi\, dx \\
		& \quad + \int_\Omega   \frac\eps 2  |\na \phi^\eps |^2 \div \xi\, dx  \\
		& \quad +  \frac{1}{\eps} \int_\Omega W_\sigma(\sigma\phi^\eps) \div  \xi\, dx  \\
		& \quad +  \frac{1}{\eps} \int_\Omega  (\rho_{\phi^\eps} - \sigma\phi^\eps) \na \phi^\eps \cdot \xi - W_\sigma(\sigma\phi^\eps)\div\xi \, dx.  
	\end{aligned}
\end{equation}
The rest of the proof consists in passing to the limit $\varepsilon \to 0^+$ in \eqref{eq:eqre} using the assumption of convergence of the energy \eqref{eq:EA} and its consequences.
But first, we point out that the first four terms in \eqref{eq:eqre} can all be bounded by the energy and energy dissipation and are expected to have a limit. This implies that the last term is also bounded (and should have a limit). 
More precisely, \eqref{eq:eqre} implies,  for all $t \in [0,T]$, the following bound for the last term:
\begin{equation}\label{eq:bdgf}
	\begin{aligned}
		&\left|\frac{1}{\eps} \int_\Omega   (\rho_{\phi^\eps(t)} - \sigma\phi^\eps(t)) \na \phi^\eps(t) \cdot \xi(t) - W_\sigma(\sigma\phi^\eps(t))\div\xi(t) \, dx  \right|  \\
		&\qquad\leq \|\xi(t)\|_{L^\infty(\Omega)^d}
		\left(   \int_\Omega  \eps |\pa_t\phi^\eps(t)|^2   \, dx  \right)^{1/2}\left( \int_\Omega  \eps | \na \phi^\eps(t) |^2\, dx \right)^{1/2} \\
		&\qquad\qquad+ C \|D\xi(t)\|_{L^\infty(\Omega)^{d\times d}} \int_\Omega  \eps | \na \phi^\eps(t)|^2\, dx+   \|\div \xi(t)\|_{ L^\infty(\Omega)}   \frac{1}{\eps}  \int_\Omega W_\sigma (\sigma\phi^\eps(t)) \, dx  \\
		&\qquad\leq \|\xi(t)\|_{L^\infty(\Omega)^d} \sD_\eps(t)^{1/2}\big(2\sJ_\eps(\phi_{\init}^\varepsilon)\big)^{1/2}    
		+C \|D\xi(t)\|_{L^{\infty}(\Omega)^{d \times d}} \sJ_\eps(\phi_{\init}^\varepsilon). 
	\end{aligned}
\end{equation}

The next proposition shows that the first four terms in \eqref{eq:eqre} have a limit.
The last term, which gives rise to the Lagrange multiplier, will be handled by Proposition \ref{prop:LM}.
\begin{proposition}\label{prop:convterm}
	Under the assumption of convergence of the energy \eqref{eq:EA} and for any test vector field $\xi \in [C^1_c((0,T)\times \Omega)]^d$, the following 
	convergences hold (for the same subsequence as in Lemma \ref{lem:cv}):
	\begin{equation}\label{eq:conv1}
		\lim_{n\to\infty}  \int_0^T\!\!\!\!\int_\Omega   \frac{\eps_n} 2  |\na \phi^{\eps_n} |^2 \div \xi\, dx \, dt
		= \gamma\int_0^T\!\!\!\!\int_\Omega   \frac 1 2     \div \xi  |\na \phi| \, dt, 
	\end{equation}

	\begin{equation}\label{eq:conv2}
		\lim_{n\to\infty}  \frac{1}{\eps_n} \int_0^T\!\!\!\!\int_\Omega W_\sigma (\sigma \phi^{\eps_n}) \div  \xi\, dx \, dt
		= \gamma\int_0^T\!\!\!\!\int_\Omega   \frac 1 2     \div \xi  |\na \phi| \, dt, 
	\end{equation}
	
	\begin{equation}\label{eq:conv3}
		\lim_{n\to\infty} \int_0^T\!\!\!\!\int_\Omega  \eps_n \na \phi^{\eps_n}\otimes \na \phi^{\eps_n}:D\xi\, dx \, dt 
		=\gamma\int_0^T\!\!\!\!\int_\Omega \nu \otimes  \nu :D\xi\, |\na \phi| \, dt, 
	\end{equation}

	\begin{equation}\label{eq:conv4}
		\lim_{n\to\infty} \int_0^T\!\!\!\! \int_\Omega  {\eps_n} \pa_t\phi^{\eps_n} \na \phi^{\eps_n} \cdot \xi\, dx \, dt
		=  \gamma\int_0^T\!\!\!\!\int_\Omega  V  \xi\cdot \na \phi \, dt.
	\end{equation}

\end{proposition}
In the right-hand side of \eqref{eq:conv3}, $\nu(t)$ is the Radon-Nikodym derivative of the Radon measure $\nabla \phi(t)$ on $\cl\Omega$ w.r.t.\ its total variation $\abs{\nabla\phi(t)}$. It coincides with the measure-theoretic inner unit normal vector to the reduced boundary $\partial^\ast E(t)$, where $E(t) = \supp \phi^0(t)$ is a set of finite perimeter in $\Omega$, as shown in Proposition \ref{prop:conv}. In particular, for a.e. $t \in [0,T]$, we have  $\abs{\nu(t,x)} = 1$ for $\abs{\nabla\phi(t)}$-a.e.\ $x$.

\begin{proof}
	
	The first two limits, \eqref{eq:conv1} and \eqref{eq:conv2}, follow, respectively, from the second and third limits in \eqref{eq:EA1} of Lemma \ref{lem:cv} and Lebesgue dominated convergence theorem.
	The third limit \eqref{eq:conv3} is handled as in the proof of Proposition 5.2 in \cite{KMW2}. It is a particular case of Reshetnyak's continuity theorem; we refer to 
	\cite[Proposition A.3]{KMW2} for a simple proof. 
	This proof can in fact be adapted to prove \eqref{eq:conv4}, as we show below. Though written a bit differently, this forthcoming argument is essentially similar to the one used in \cite[Section 3.3]{Laux3}. 
	
	Given a vector field $h\in L^\infty(0,T;[C^1_c(\Omega)]^d)$ such that $|h(t,x)|\leq 1$ for all $t,x$, we write
	$$
	\left| \int_0^T    \int_\Omega  {\eps_n} \pa_t\phi^{\eps_n} \na \phi^{\eps_n} \cdot \xi\, dx \, dt
	-  \gamma \int_0^T   \int_\Omega  V  \xi\cdot \na \phi \, dt\right| \leq A_n + B_n + E
	$$
	with
	\begin{align*}
		A_n & :=\left| \int_0^T   \int_\Omega  {\eps_n} \pa_t\phi^{\eps_n} \na \phi^{\eps_n} \cdot \xi\, dx \, dt
		- \int_0^T   \! \int_\Omega  {\eps_n} \pa_t\phi^{\eps_n}| \na \phi^{\eps_n} | h \cdot \xi\, dx \, dt \right|,\\
		B_n&  :=\left| \int_0^T    \int_\Omega  {\eps_n} \pa_t\phi^{\eps_n}| \na \phi^{\eps_n} | h \cdot \xi\, dx \, dt - 
	 \gamma\int_0^T   \int_\Omega  V  h\cdot\xi | \na \phi |  \, dt\right|,\\
		E& :=\gamma \left|  \int_0^T   \int_\Omega  V  h\cdot\xi | \na \phi |  \, dt - \int_0^T \int_\Omega  V  \xi\cdot \na \phi \, dt\right|.
	\end{align*}
	The idea behind this splitting is to compare the direction of the approximate unit normal vector (or its limit) with some arbitrary but frozen direction (encoded by the vector field $h$). The task is then to unfreeze this direction and show it can be chosen in a manner that approximates arbitrarily well the direction of the limiting unit normal vector.
	
	We introduce some notation for the approximate unit normal vector $\nu^\eps :=\frac{\na \phi^\eps}{|\na \phi^\eps|}$ (and defined to be some arbitrary member of $\bS^{d-1}$ whenever $\abs{\nabla \phi^\varepsilon} = 0$), and we estimate
	\begin{align*}
		A_n & = \left| \int_0^T    \int_\Omega \eps_n \pa_t \phi^{\eps_n} |\na \phi^{\eps_n}| \xi \cdot (\nu^{\eps_n} - h) \, dx \, dt\right|\\
		&\leq \delta \int_0^T\int_\Omega \eps_n |\pa_t \phi^{\eps_n} |^2   \, dx \, dt
		+ \frac {\|\xi\|_{L^\infty(0,T; L^\infty(\Omega)^d)}}{ \delta} \int_0^T\int_\Omega \eps_n |\na \phi^{\eps_n}|^2 |\nu^{\eps_n} - h|^2 \, dx \, dt\\
		& \leq \delta \sJ_{\eps_n}(\phi_\init^{\varepsilon_n})
		+2 \frac {\|\xi\|_{L^\infty(0,T; L^\infty(\Omega)^d)}}{ \delta} \int_0^T\int_\Omega \eps_n |\na \phi^{\eps_n}|^2 \left(1 - \nu^{\eps_n} \cdot  h\right) \, dx \, dt,
	\end{align*}
	where we used the energy dissipation inequality \eqref{eq:energy} and the fact that $|\nu^\eps - h|^2 = |\nu^\eps|^2 -2 \nu^\eps \cdot  h+ |h|^2 \leq 2-2\nu^\eps \cdot  h$ (since $|\nu^\eps|=1$ and $|h|\leq 1$).
	Furthermore, equipartition of the energy permits us to replace $\varepsilon_n\abs{\nabla \phi^{\varepsilon_n}}^2$ by $\abs{\nabla F_\sigma(\sigma\phi^{\varepsilon_n})}$ up to some asymptotically negligible error (see \eqref{eq:equipartitionModicaMortola}):
	\begin{align*}
		\int_0^T\int_\Omega \eps_n |\na \phi^{\eps_n}|^2 \left(1 - \nu^{\eps_n} \cdot  h\right) \, dx \, dt
		& =\int_0^T\int_\Omega  |\na F_\sigma(\sigma \phi^{\eps_n})| \left(1 - \nu^{\eps_n} \cdot  h\right) \, dx \, dt + o(1)\\
		& = \int_0^T\int_\Omega  |\na \psi^{\varepsilon_n}|   \, dx \, dt-\int_0^T\int_\Omega  \na \psi^{\varepsilon_n} \cdot h \, dx \, dt +o(1).
	\end{align*}
	Above, we reintroduced the auxiliary function $\psi^{\varepsilon_n} :=F_\sigma(\sigma \phi^{\eps_n})$ and used the observation that $\nu^{\eps_n} = \frac{\na \psi^{\varepsilon_n}}{|\na \psi^{\varepsilon_n}|}$. The first integral on the second line has a limit (using \eqref{eq:lim1}) and (after an integration by parts) the second integral has a limit (using Proposition \ref{prop:conv})). Thus, 
	$$
	\limsup_{n\to\infty} A_n \leq 
	\delta \sup_{n\in\bN}\sJ_{\eps_n}(\phi_{\init}^{\varepsilon_n})
	+2 \frac {\|\xi\|_{L^\infty(0,T; L^\infty(\Omega)^d)}}{ \delta}\left( 
	\gamma \int_0^T\int_\Omega  |\na \phi|    \, dt+\int_0^T\int_\Omega \gamma \phi \, \div h \, dx \, dt \right).
	$$
	We recall that the well-preparedness assumption on the initial data guarantees that the energy of the initial data is uniformly bounded.
	\medskip
	
	We can  bound the term $E$ similarly: Replacing $\nabla\phi(t)$ with $\nu(t)\abs{\nabla \phi(t)}$ 
	and using a similar argument as above leads to 
	\begin{align*}
	E 
	& =  \gamma \left|  \int_0^T   \int_\Omega  V  (h-\nu) \cdot\xi | \na \phi |  \, dt \right|\\
	& \leq 
	\delta \gamma  \int_0^T \int_\Omega V^2 |\na \phi| \, dt 
	+2  \gamma^2 \frac {\|\xi\|_{L^\infty(0,T; L^\infty(\Omega)^d)}}{ \delta}\left( 
	\int_0^T\int_\Omega  |\na \phi|    \, dt+\int_0^T\int_\Omega  \phi \, \div h \, dx \, dt \right).
	\end{align*}
	
	\medskip

	Finally, we claim that $\lim_{n\to\infty} B_n=0$. Indeed,  \eqref{eq:L2diff} of Lemma \ref{lem:cv}
	and the fact that $\eps_n^{1/2} \pa_t \phi^{\eps_n}$ is bounded in $L^2$ 
	 permits us to replace $\varepsilon_n^{1/2}\abs{\nabla\phi^{\varepsilon_n}}$ with $\varepsilon_n^{-1/2}\sqrt{2W_\sigma(\sigma\phi^{\varepsilon_n})}$. We then introduce the auxiliary function $\psi^{\varepsilon_n} :=F_\sigma(\sigma\phi^{\varepsilon_n})$ and recall $F'_\sigma(v) = \frac{1}{\sigma} \sqrt{2W_\sigma(v)}$, and by using \eqref{eq:convergenceOfVelocity} from Lemma \ref{lem:vel}, we find
	\begin{align*}
		\lim_{n\to\infty } \int_0^T    \int_\Omega  {\eps_n} \pa_t\phi^{\eps_n}| \na \phi^{\eps_n} | h \cdot \xi\, dx \, dt  
		& = \lim_{n\to\infty } \int_0^T    \int_\Omega   \pa_t\phi^{\eps_n} \sqrt{2W_\sigma(\sigma \phi^{\eps_n})} h \cdot \xi\, dx \, dt  \\
		& = \lim_{n\to\infty } \int_0^T    \int_\Omega   \pa_t\psi^{\varepsilon_n}   h \cdot \xi\, dx \, dt  \\
		& =  \gamma\int_0^T    \int_\Omega  V h \cdot \xi\, |\na \phi|  \, dt 
	\end{align*}
hence 	$\lim_{n\to\infty} B_n=0$.

\medskip

	Putting things together, we proved
	\begin{gather*}
		\limsup_{n\to\infty}
		\left| \int_0^T    \int_\Omega  {\eps_n} \pa_t\phi^{\eps_n} \na \phi^{\eps_n} \cdot \xi\, dx \, dt
		- \gamma  \int_0^T   \int_\Omega  V  \xi\cdot \na \phi \, dt\right| \\
		\leq C\delta + \frac{C}{\delta} \left( 
		\int_0^T\int_\Omega  |\na \phi|   \, dx \, dt+\int_0^T\int_\Omega  \phi \, \div h \, dx \, dt \right).
	\end{gather*}
	Since the first line is independent of $h$, we can use the definition of the $BV(\Omega)$-seminorm of $\phi$ to make the $\cO(\delta^{-1})$ term on the second line arbitrarily small, which leads to
	$$
	\limsup_{n\to\infty}
	\left| \int_0^T    \int_\Omega  {\eps_n} \pa_t\phi^{\eps_n} \na \phi^{\eps_n} \cdot \xi\, dx \, dt
	-   \gamma \int_0^T   \int_\Omega  V  \xi\cdot \na \phi \, dt\right| 
	\leq C\delta.
	$$
	We can now let $\delta\to0^+$ to get the result. 
\end{proof}

The final task is to 
identify the limit of the last term in \eqref{eq:eqre} and 
show that the deviation between the our nonlocal reaction term and the standard local reaction term converges weakly to the Lagrange multiplier $\Lambda = \Lambda(t)$ that is responsible for the volume preservation of the limiting flow.
We will make use of the particular form of the term $\rho_{\phi^{\eps_n}}$ given by Proposition \ref{prop:rho}-(ii) and (iii).

\begin{proposition}[Existence of the Lagrange multiplier]\label{prop:LM}
	Under the assumption of convergence of the energy \eqref{eq:EA},
	there exists $\Lambda \in L^2(0,T)$ such that
	\begin{equation}\label{eq:conv5}
		\lim_{n\to\infty}  \frac{1}{\eps_n} \int_0^T\int_\Omega  (\rho_{\phi^{\eps_n}} - \sigma\phi^{\eps_n}) \na \phi^{\eps_n} \cdot \xi - W_\sigma(\sigma \phi^{\eps_n})\div\xi \, dx \, dt  = - \int_0^T \Lambda (t)\int_\Omega    \gamma \phi\,  \div  \xi \, dt 
	\end{equation}
	for all $\xi \in [C^1([0,T] \times \cl\Omega)]^d$  satisfying $\xi\cdot n=0$ on $\pa\Omega$, where $n$ is the unit normal vector to $\partial\Omega$.
\end{proposition}

\begin{proof}[Proof of Proposition \ref{prop:LM}] The proof is divided in several steps.
	\medskip
	
	\noindent{\bf Step 1:} First we are going to show that for all $t\in[0,T]$ there exists a constant $ \Lambda^{\eps_n}(t) \in \bR$ such that
	\begin{equation}\label{eq:lambda0}
		\begin{aligned}
			\frac{1}{\eps_n} \int_\Omega  (\rho_{\phi^{\eps_n}(t)} - \sigma\phi^{\eps_n}(t)) \na \phi^{\eps_n}(t) \cdot \xi(t) - W_\sigma(\sigma \phi^{\eps_n}(t))\div\xi(t) \, dx 
			\\
			\qquad=   -  \Lambda^{\eps_n}(t)  \int_\Omega  \frac{ \gamma }{\sigma}
			\rho_{\phi^{\eps_n}(t)} \div \xi(t)  \, dx + \cO \big(\| \div\xi(t) \|_{L^\infty(\Omega)} z^{\eps_n}(t) \big),
		\end{aligned}
	\end{equation}
	where the function $t \mapsto z^{\eps_n}(t)$ (defined in \eqref{eq:zeps}) was the deviation between the energy $\sJ_{\varepsilon_n}(\phi^{\varepsilon_n}(t))$ and the optimal lower bound on the Modica-Mortola functional, and vanishes strongly in $L^1(0,T)$ as $n \to \infty$.

	The convergence of the energy assumption \eqref{eq:EA} implies that the standard local reaction term and our nonlocal reaction term are asymptotically close in $L^1((0,T)\times\Omega)$; this is \eqref{eq:Wg} of Lemma \ref{lem:cv}, and the error is measured in $L^1(\Omega)$ pointwise in time by the second inequality in \eqref{eq:lim3}. Using this and the definition \eqref{eq:doubleWellPotential-Density} of $W$, we obtain (after an integration by parts)
	\begin{gather*}
		\frac{1}{\eps_n}  \int_\Omega   W_\sigma(\sigma \phi^{\eps_n}(t))\div\xi(t) \, dx \\
		= 
		\frac{1}{\eps_n} \int_\Omega   \left[f(\rho_{\phi^{\eps_n}(t)}) +  a \rho_{\phi^{\eps_n}(t)} -\rho_{\phi^{\eps_n}(t)} \phi^{\eps_n}(t)+ \frac{\sigma}{2} |\phi^{\eps_n}(t)|^2 \right] \div\xi(t) \, dx   + \mathcal O (\| \div\xi(t) \|_{L^\infty_x} z^{\eps_n}(t) )\\
		= 
		\frac{1}{\eps_n} \int_\Omega   \left[f(\rho_{\phi^{\eps_n}}(t)) +  a \rho_{\phi^{\eps_n}(t)} -\rho_{\phi^{\eps_n}(t)} \phi^{\eps_n}(t)   \right] \div\xi(t) -  \sigma\phi^{\eps_n}(t) \na \phi^{\eps_n}(t) \cdot \xi(t)  \, dx + \mathcal O (\| \div\xi(t) \|_{L^\infty_x} z^{\eps_n}(t)).
	\end{gather*}
	Recall also that $\phi^\varepsilon(t) = f'(\rho_{\phi^\varepsilon(t)}) + \ell^\varepsilon(t)$ a.e.\ on $\set{\rho_{\phi^\varepsilon(t)} > 0}$ (see Proposition \ref{prop:rho}-(ii)), where $\ell^\varepsilon(t) \in \bR$ is the Lagrange multiplier for the mass constraint on $\rho_{\phi^\varepsilon(t)}$. Thus, after multiplying this equality by $\rho_{\phi^\varepsilon(t)}$, this equality is valid a.e.\ on $\Omega$, so the above becomes
	\begin{gather*}
		= \frac{1}{\varepsilon_n}\int_\Omega \big[f(\rho_{\phi^{\varepsilon_n}(t)}) - \rho_{\phi^{\varepsilon_n}(t)}f'(\rho_{\phi^{\varepsilon_n}(t)})  \big]\div\xi(t) - \sigma\phi^{\varepsilon_n}(t) \nabla \phi^{\varepsilon_n}(t) \cdot \xi(t) \,dx \\
		+ \frac{a - \ell^{\varepsilon_n}(t)}{\varepsilon_n}\int_\Omega\rho_{\phi^{\varepsilon_n}(t)} \div\xi(t)  \,dx + \cO(\norm{\div\xi(t)}_{L^\infty_x} z^{\varepsilon_n}(t)).
	\end{gather*}
	Substituting this into the left-hand side of \eqref{eq:lambda0} gives 
	\begin{gather*}
		\frac{1}{\eps_n}  \int_\Omega  (\rho_{\phi^{\eps_n}(t)} - \sigma\phi^{\eps_n}(t)) \na \phi^{\eps_n}(t) \cdot \xi(t) - W_\sigma(\sigma \phi^{\eps_n}(t))\div\xi(t) \, dx \\
		= 
		\frac{1}{\eps_n}
		\int_\Omega  \rho_{\phi^{\eps_n}(t)}   \na \phi^{\eps_n}(t) \cdot \xi(t) + \nabla \big[f(\rho_{\phi^{\varepsilon_n}(t)}) - \rho_{\phi^{\varepsilon_n}(t)}f'(\rho_{\phi^{\varepsilon_n}(t)})  \big] \cdot \xi(t) \,dx \\
		- \frac{a - \ell^{\varepsilon_n}(t)}{\varepsilon_n}\int_\Omega \rho_{\phi^{\varepsilon_n}(t)} \div\xi(t) \,dx 
		+ \mathcal O (\| \div\xi(t) \|_{L^\infty_x} z^{\eps_n}(t) ).
	\end{gather*}
	From the Euler-Lagrange equation satisfied by $\rho_{\phi^\varepsilon(t)}$ for each $t$ (see Proposition \ref{prop:rho}-(iii)), we see that the integral on the second line vanishes. We thus obtain \eqref{eq:lambda0} with
	$$ \Lambda^{\eps_n}(t) :=  \frac{\sigma}{\gamma} \frac{ a - \ell^{\eps_n}(t)   }{  \eps_n}.$$
	\medskip

	\noindent{\bf Step 2:} Ultimately, we aim to pass to the limit in \eqref{eq:lambda0}, and for that we want to show that 
	$\pseq{\Lambda^{\varepsilon_n}}_n$ is bounded in $L^2(0,T)$.
This will be proved by using \eqref{eq:lambda0} and showing that for  sufficiently small $\varepsilon_n$ and a well-chosen vector field $x \mapsto \xi(x)$,  the integral $\int_\Omega \rho_{\phi^{\varepsilon_n}(\cdot)}(x)\div\xi(x)\,dx$ can be bounded below  (locally uniformly in $t$).
	
	In this step, we first show that 
	there exists $\eta > 0$ such that for each $t_0 \in [0,T]$
	 there exists $\xi_{t_0} \in C^{1,\alpha}(\Omega; \bR^d)$, for some $\alpha \in (0,1)$, such that 
	\begin{equation}\label{eq:xi0}
	\int_\Omega \sigma\phi(t_0) \div\xi_{t_0} \,dx \ge \frac{\eta}{2}.
	\end{equation}
	
	Indeed, given $t_0 \in [0,T]$ we take $\pseq{\psi_{k}}_k \subset C^\infty(\cl\Omega)$ (which depends on $t_0$)  such that $ \int_\Omega \psi_{k} \,dx = 0$ and 
	\begin{gather*}
		\lim_{k\to\infty} \psi_{k} = \sigma \phi(t_0) - \frac{1}{\abs{\Omega}}	\quad \text{strongly in } L^1(\Omega).
	\end{gather*}
(note that this limit has integral zero over $\Omega$).
Since $\sigma\phi(t_0)=\theta \chi_{E(t_0)} \in L^\infty(\Omega)$ and $\pseq{\psi_{k}}_k$ converges strongly in $L^1(\Omega)$, we have the following convergence: 
	\begin{gather*}
		\lim_{k\to\infty} \int_\Omega \sigma\phi(t_0) \psi_{k} \,dx = \int_\Omega \sigma\phi(t_0) \big( \sigma\phi(t_0) - \frac{1}{\abs{\Omega}} \big)\,dx = \theta\int_{E(t_0)} \theta \chi_{E(t_0)} - \frac{1}{\abs{\Omega}} \,dx \\
		= \theta \big(\theta\abs{E(t_0)} - \frac{\abs{E(t_0)}}{\abs{\Omega}}\big) = \theta\big(1 - \frac{1}{\theta\abs{\Omega}}\big) > 0,
	\end{gather*}
	where the final inequality follows from the assumption \eqref{eq:st} that $\theta\abs{\Omega}>1$. Set
	\begin{equation} \label{eq:lowerBoundConstant}
		\eta :=\theta\big(1 - \frac{1}{\theta\abs{\Omega}}\big) > 0,
	\end{equation}
	which we note is independent of $t_0$.
	
	To find the corresponding  vector fields $\xi$, we solve the potential problems:
	\begin{equation}\label{eq:potentialProblemLagrangeMultiplier}
		\begin{cases}
			\lap u_{k} = \psi_{k}	&	\text{in } \Omega,\\
			\nabla u_{k} \cdot n = 0	&	\text{on } \partial\Omega.
		\end{cases}
	\end{equation}
	The solution to the above exists since $\int_\Omega\psi_{k} \,dx = 0$.
	 Classical Schauder estimates imply $\xi_{k} :=\nabla u_{k} \in C^{1,\alpha}(\Omega)$, and by construction  
$\div \xi_{k} = \psi_{k}$ on $\Omega$, $\xi_{k} \cdot n = 0$ on $\partial\Omega$, and 
	\[
	\lim_{k \to \infty} \int_\Omega \sigma\phi(t_0) \div\xi_{k} \,dx = \eta.
	\]
	We thus define $\xi_{t_0} :=\xi_k$ (and $\psi_{t_0} :=\div  \xi_{t_0} = \psi_k$) for some $k$ large enough for which  \eqref{eq:xi0} holds.
	
	\medskip
	\noindent{\bf Step 3:} The main concern with \eqref{eq:xi0} is that $\xi$ depends on $t_0$. But, as we will show below, the continuity of $\phi$ implies that we can use the same $\xi_{t_0}$ for all $t$ in some small interval around $t_0$.
	
Indeed, we have $\phi \in C^{1/4}([0,T]; L^1(\Omega))$  (see Proposition \ref{prop:conv}) and so
	\begin{align*}
		\int_\Omega \sigma\phi(t) \div\xi_{t_0} \,dx &= \int_\Omega \sigma\phi(t_0) \div\xi_{t_0} \,dx + \int_\Omega \sigma(\phi(t) - \phi(t_0)) \psi_{t_0} \,dx \\
		&\ge \frac{\eta}{2} - \sigma\norm{\phi(t) - \phi(t_0)}_{L^1(\Omega)} \norm{\psi_{t_0}}_{L^{\infty}(\Omega)}	\\
		&\ge \frac{\eta}{2} - C \norm{\psi_{t_0}}_{L^\infty(\Omega)} |t-t_0|^{1/4}.
	\end{align*}
	We  can choose $\psi_{t_0}$ from Step 2 satisfying $\norm{\psi_{t_0}}_{L^\infty(\Omega)} \le \max\left\{\abs{\theta - \abs{\Omega}^{-1}}, \, \abs{\Omega}^{-1}\right\}$.
	Thus, we can choose $\delta > 0$, independent of $t_0$, 
	such that $C \norm{\psi_{t_0}}_{L^\infty(\Omega)} |t-t_0|^{1/4}\leq \eta/4$ when $|t-t_0|\leq \delta$ and get:
	\[
	\forall \, t \in (t_0 - \delta, t_0 + \delta),	\qquad	\int_\Omega \sigma\phi(t) \div \xi_{t_0} \,dx \ge \frac{\eta}{4}.
	\]

	\medskip 
	\noindent{\bf Step 4:} 
	The strong convergence of $\rho_{\phi^{\varepsilon_n}(t)} $ to   $\sigma\phi(t) $ in $L^\infty(0,T;L^1(\Omega))$ (Proposition \ref{prop:conv}) implies that there exists $\eps_0>0$ such that 
	\[
	\forall \, t \in (t_0 - \delta, t_0 + \delta),\quad \forall \eps_n\leq \eps_0	\qquad \int_\Omega \rho_{\phi^{\varepsilon_n}(t)} \div \xi_{t_0} \,dx \ge \frac{\eta}{8}
	\] 
(once again, the fact that $\div \xi_{t_0}  = \psi_{t_0} $ is bounded in $L^\infty$ uniformly in $t_0$ implies that $\eps_0$ does not depend on $t_0$).

	\medskip
	\noindent{\bf Step 5:} Since $\seq{(t_0 - \delta, t_0+\delta)}_{t_0 \in [0,T]}$ is an open cover of $[0,T]$, by compactness there exists a finite open subcover $\seq{(t_i - \delta, t_i + \delta)}_{i=1}^N$. Suppose $\varepsilon_n < \varepsilon$, where $\varepsilon$ is from Step 4, and fix $i \in \set{1,\dots,N}$ and $t \in (t_i - \delta, t_i + \delta)$. On the one hand, Step 4 gives 
	\[
\abs{\Lambda^{\varepsilon_n}(t)  \int_\Omega \rho_{\phi^{\varepsilon_n}(t)} \div\xi_{t_i} \,dx } =	\abs{\Lambda^{\varepsilon_n}(t)} \int_\Omega \rho_{\phi^{\varepsilon_n}(t)} \div\xi_{t_i} \,dx \ge \frac{\eta}{8}\abs{\Lambda^{\varepsilon_n}(t)}.
	\]
	On the other hand, from the equality \eqref{eq:lambda0} in Step 1 and the calculation \eqref{eq:bdgf} (with $\xi_{t_i}$ in place of $\xi(t)$)
	\[
	\abs{\Lambda^{\varepsilon_n}(t)\int_\Omega\rho_{\phi^{\varepsilon_n}(t)} \div\xi_{t_i} \,dx} \le C\norm{\psi_{t_i}}_{L^\infty(\Omega)} z^{\varepsilon_n}(t) + C\norm{\xi_{t_i}}_{L^\infty(\Omega)^d} \sD_{\varepsilon_n}^{1/2}(t) + C \norm{D\xi_{t_i}}_{L^\infty(\Omega)^{d \times d}},
	\]
	where we have used that the initial energies are bounded. 
	We recall that for each $t_i$,  $\xi_{t_i} \in C^{1,\alpha}(\Omega)^d$ (although the corresponding bound is not uniform in $t_i$)
	and so
	$$
	 \frac{\eta}{8}\abs{\Lambda^{\varepsilon_n}(t)} \leq 
	 K\left(   z^{\varepsilon_n}(t) +  \sD_{\varepsilon_n}^{1/2}(t) + C \right)\leq K\left(   \sJ_{\eps_n} (\phi^{\eps_n}(t)) +  \sD_{\varepsilon_n}^{1/2}(t) + C \right)
$$
where 
$$K =  \sup_{i=1,\dots ,N} \norm{\xi_{t_i}}_{C^{1,\alpha}(\Omega)^{d }}<\infty.$$
	
We deduce (using the energy dissipation inequality \eqref{eq:energy} and the bound on the energy of the initial data)
 
	\[
	\int_0^T\abs{\Lambda^{\varepsilon_n}(t)}^2 \,dt \le  C K
	\]
	
	We thus extract a (not relabled) subsequence along which $\Lambda^{\varepsilon_n} \rightharpoonup \Lambda$ weakly in $L^2(0,T)$. Since $\rho_{\phi^{\varepsilon_n}} \to \sigma\phi$ strongly in $L^\infty(0,T; L^1(\Omega))$, we have, for any test function $\xi \in C^1([0,T] \times \cl\Omega)$, that 
	\[
	\lim_{n \to \infty}\int_\Omega \rho_{\phi^{\varepsilon_n}} \div \xi \,dx = \int_\Omega \sigma\phi \,\div\xi \,dx \quad \text{strongly in } L^\infty(0,T),
	\]
	so then the product $\Lambda^{\varepsilon_n}\int_\Omega \rho_{\phi^{\varepsilon_n}}\div\xi \,dx$ converges strongly in $L^2(0,T)$ to $\Lambda(t)\int_\Omega \sigma\phi\div\xi\,dx$.
	
	In view of \eqref{eq:lambda0}, this completes the proof of Proposition \ref{prop:LM}.		
\end{proof}

Putting everything together, we now easily get:
\begin{proof}[Proof of Theorem \ref{thm:main} - (ii)]
We simply pass to the limit in \eqref{eq:eqre}, which we recall here:
$$
	\begin{aligned}
		\int_0^T \int_\Omega  \eps_n \pa_t\phi^{\eps_n} \na \phi^{\eps_n} \cdot \xi\, dx \, dt
		& = - \int_0^T\int_\Omega  \eps_n \na \phi^{\eps_n}\otimes \na \phi^{\eps_n}:D\xi\, dx\, dt \\
		& \quad + \int_0^T\int_\Omega   \frac{\eps_n}{2}  |\na \phi^{\eps_n} |^2 \div \xi\, dx \,dt \\
		& \quad +  \frac{1}{\eps_n}\int_0^T \int_\Omega W_\sigma(\sigma\phi^{\eps_n}) \div  \xi\, dx\,dt  \\
		& \quad +  \frac{1}{\eps_n}\int_0^T \int_\Omega  (\rho_{\phi^{\eps_n}} - \sigma \phi^{\eps_n}) \na \phi^{\eps_n} \cdot \xi - W_\sigma(\sigma\phi^{\eps_n})\div\xi \, dx\, dt.  
	\end{aligned}
$$
Using Proposition \ref{prop:convterm} (for the first four terms) and Proposition \ref{prop:LM} (for the last one), we get:
$$
	\begin{aligned}
		\gamma \int_0^T\!\!\!\!\int_\Omega  V  \xi\cdot \na \phi \, dt.
		& = - \gamma\int_0^T\!\!\!\!\int_\Omega \nu \otimes  \nu :D\xi\, |\na \phi| \, dt  + \gamma\int_0^T\!\!\!\!\int_\Omega       \div \xi  |\na \phi| \, dt   - \int_0^T \Lambda (t)\int_\Omega   \gamma \phi \, \div  \xi \, dt,
	\end{aligned}
$$
and the proof is complete.
\end{proof}

\bibliographystyle{plain}
\bibliography{mybib}

\end{document}